\documentclass[a4paper,11pt]{amsart}

\usepackage[T1,T2A]{fontenc}
\usepackage[russian,english,]{babel}
\usepackage{amsmath}
\usepackage{amsfonts}
\usepackage{amssymb}
\usepackage{mathrsfs}
\usepackage{amsthm}
\usepackage{graphicx}
\usepackage{color}
\graphicspath{{./Pictures/}}

\newcommand{\df}{\buildrel\mathrm{def}\over=}
\newcommand{\Bell}{\boldsymbol{B}}

\newcommand{\LL}{{\scriptscriptstyle\mathrm L}}
\newcommand{\eps}{\varepsilon}

\newcommand{\const}{\mathrm{const}}
\DeclareMathOperator{\sign}{sign}
\newcommand{\BMO}{\mathrm{BMO}}

\newcommand{\half}{\tfrac12}
\newcommand{\eq}[2]{\begin{equation}\label{#1} #2 \end{equation}}

\newtheorem{Le}{Lemma}[section]
\newtheorem{Prop}[Le]{Proposition}
\newtheorem{Th}[Le]{Theorem}

\theoremstyle{definition}

\numberwithin{equation}{section}

\newcommand{\ut}[1]{^{\scriptscriptstyle\text{\rm #1}}}

\newcommand{\Oe}{\Omega_\eps}
\newcommand{\Om}[1]{\Omega\ut{#1}_\eps}

\newcommand\mytitle{Some extremal problem for martingale transforms. \rm{III}}
\newcommand\mytitleshort{Extremal problem for martingale transforms}

\begin{document}
\title[\mytitleshort]{\mytitle}

\author{V.~I.~Vasyunin}

\subjclass[2020]{Primary 42B35, 60G46}
\keywords{Bellman function, martingale transform, 
diagonally concave function}

\begin{abstract}
This paper is a direct continuation of papers~\cite{First} and~\cite{Second}. By this reason, neither the introductory part of the paper nor the bibliography are duplicated here. However, for the reader's convenience, the formulas from the previous papers that are referenced are given in a special addendum at the end of the paper under their original numbers.

In this paper, we investigate new local foliations: major and medium pockets and several bilinearity regions. The appearance of such local foliations is illustrated by further analysis of examples with third-degree polynomials. At the end of the paper, a complete overview of possible foliations is given when the boundary values are third-degree polynomials.
\end{abstract}

\thanks{This work is supported by the Russian Science Foundation, grant number 24-71-10011 (https://rscf.ru/project/24-71-10011/).} 
\maketitle

\setcounter{section}{12}
\setcounter{figure}{24}

\section{Once more on horizontal herringbones}

We already know quite a lot about horizontal herringbones. The general theory was described in Section~4 of~\cite{First}, then applied to special cases in Sections~5 and~6 of the same paper and in Section~8 of the next paper~\cite{Second}. Our knowledge has been enriched by the analysis of examples with boundary conditions in the form of third-degree polynomials.

Now, beginning the third paper in this series, we attempt to identify some further general properties of horizontal herringbones that will allow us, on the one hand, to avoid repeating the same arguments when considering specific examples, and, on the other hand, to understand which parameters of the boundary values determine the foliation in the general case.

In this section, we consider the question of at which boundary points horizontal herringbones may terminate. Clearly, there are four symmetric cases: (1) the left herringbone terminates on the upper boundary, (2) the left herringbone terminates on the lower boundary, (3) the right herringbone terminates on the upper boundary, and (4) the right herringbone terminates on the lower boundary. We shall, of course, study in detail only one case, namely the first of the listed ones; however, for convenience, we provide the statements of the assertions for all four cases, as is done, for example, in Propositions~6.1--6.4.

\begin{Th}
\label{160310}
For a point $(u+\eps,\,\eps)$ to be a tip of a left herringbone in the strip $\Oe,$ the following conditions must be satisfied\textup:
\begin{enumerate}
\item
\label{140301}
$D_+(u,\eps)=0;$
\item
\label{140302}
$f''_+(u+\eps)-f''_-(u-\eps)-2\eps f'''_+(u+\eps)\ge0;$
\item
\label{140303}
$f'''_+(u+\eps)\ge0.$
\end{enumerate}
\end{Th}

\begin{proof}
As usual for the study of left herringbones, we consider field~(5.3), one of whose integral curves, when shifted to the right by $\eps$, gives the spine of a herringbone. Since at any nonsingular point of field~(5.3) the integral curve intersects the upper boundary of the strip with slope $-1$, while the spine of any left horizontal herringbone intersecting the upper boundary must have slope in the interval from zero to one, we conclude that the point $(u,\eps)$ must be a stationary point of field~(5.3), that is, condition~\eqref{140301} holds.

The Jacobi matrix at such a point was already examined in Section~5 (see~\cite{First}). 
For the reader's convenience, we reproduce formulas (5.17) and (5.18), adopting the 
notation introduced in formula (8.20) (\cite{Second}):
\eq{150301}{
J(u,\eps)=\eps\big(f_+''(u+\eps)-f_-''(u-\eps)-2\eps f_+'''(u+\eps)\big)
\begin{pmatrix}
\phantom{i}1&1-3s
\\
-1&s-1
\end{pmatrix},
}
where $s=1+\eps\kappa_+^\LL$ and
\eq{150302}{
\kappa_+^\LL=\frac{2f_+'''(u+\eps)}{f_+''(u+\eps)-f_-''(u-\eps)-2\eps f_+'''(u+\eps)}\,.
}
We recall that, in this context, the point $(u,\eps)$ is assumed to lie on the curve $D_+=0$.

We have already computed the eigenvalues and eigenvectors of the matrix in~\eqref{150301} on several occasions; they are the roots of the polynomial $\lambda^2-s\lambda-2s$, namely
\eq{150303}{
\lambda=\frac{s\pm\sqrt{s^2+8s}}2\quad\text{and the corresponding vectors}\quad
\begin{pmatrix}
s-1-\lambda
\\
1
\end{pmatrix}.
}
We are not interested in non-real eigenvalues, since in that case $(u,\eps)$ is a spiral stationary point, and no integral curve approaching it with monotonically increasing slope in the interval $[0,1]$ can exist—--and it is precisely such curves that are of interest to us. Furthermore, we disregard the eigenvalue in which the square root is taken with the plus sign, because then $s-\lambda=\half(s-\sqrt{s^2+8s})\le0$, and $s-1-\lambda\le-1$. Thus the corresponding eigenvector has negative slope.

Thus, we shall compare the integral curve approaching the point $(u,\eps)$ with slope
\eq{150304}{
\frac1{s-1-\lambda}=\frac2{s-2+\sqrt{s^2+8s}}\,,
}
and the curve $D_+=0$, whose slope at this point equals
\eq{150305}{
\frac1{2s-1}\,.
}
The latter assertion has been verified on numerous occasions, but we shall repeat it once more. It follows simply by differentiating the identity \hbox{$2x_2D_+=0$:}
\eq{150306}{
\big[f''_+-f''_--2x_2f'''_+\big]dx_1-\big[f''_+-f''_-+2x_2f'''_+\big]dx_2=0.
}

As already mentioned, the slope of our integral curve must lie in the interval $[0,1]$, which, in view of formula~\eqref{150304}, can be rewritten as
$$
s-2+\sqrt{s^2+8s}\ge2\,.
$$
This condition is in turn equivalent to the inequality $s\ge1$, or, what is the same, $\kappa_+^\LL\ge0$.

Let us also note that for $s>1$ the strict inequality
$$
\frac2{s-2+\sqrt{s^2+8s}}>\frac1{2s-1}
$$
holds, meaning that the slope of the integral curve of interest at the point $(u,\eps)$ is strictly greater than the slope of the curve $D_+=0$. This implies that, in a neighbourhood of $(u,\eps)$, our integral curve enters the strip $\Oe$ to the right of the curve $D_+=0$.

Let us now turn to the conditions~\eqref{140302} and~\eqref{140303}, which we have left aside for the moment. Suppose that condition~\eqref{140302} is not satisfied, i.\,e., the strict inequality
\eq{150307}{
\frac{\partial D_+}{\partial x_1}(u,\eps)<0
}
holds. This implies, in particular, first, that the quantities $\kappa_+^\LL$ and $s$ are well-defined, and second, that to the right of the curve $D_+=0$ lies the domain where $D_+<0$. Consequently, the integral curve of interest passing through that domain cannot give rise to the spine of a left herringbone. Therefore, condition~\eqref{140302} must necessarily be fulfilled.

If the inequality in~\eqref{140302} is strict, then the quantity $\kappa_+^\LL$ is well-defined by formula~\eqref{150302}, and the proved inequality $\kappa_+^\LL\ge0$ is equivalent to condition~\eqref{140303}.

It remains to consider the degenerate case
\eq{230701}{
f_+''(u+\eps)-f_-''(u-\eps)-2\eps f_+'''(u+\eps)=0\,.
}
Suppose that condition~\eqref{140303} does not hold, that is, $f'''_+(u+\eps)<0$.
In this case, equality~\eqref{230701} implies the relation
\eq{150308}{
f_+''(u+\eps)<f_-''(u-\eps)\,,
}
whence it follows that
\begin{align*}
2\eps D_-(u,\eps)&=f'_+(u+\eps)-f'_-(u-\eps)-
2\eps f''_-(u-\eps)
\\
&<f'_+(u+\eps)-f'_-(u-\eps)-
2\eps f''_+(u+\eps)=2\eps D_+(u,\eps)=0\,.
\end{align*}
Thus, the necessary condition $D_-(u,\eps)\ge0$ fails, and consequently the point $(u+\eps,\eps)$ cannot be a tip of a herringbone. This verifies the necessity of condition~\eqref{140303}.
\end{proof}

We now formulate, without proof, three symmetric statements.

\begin{Th}
\label{160311}
For the point $(u+\eps,\,-\eps)$ to be a tip of a left herringbone in the strip $\Oe,$ the following conditions must be satisfied\textup:
\begin{enumerate}
\item
\label{160301}
$D_-(u,-\eps)=0;$
\item
\label{160302}
$f''_+(u-\eps)-f''_-(u+\eps)+2\eps f'''_-(u+\eps)\le0;$
\item
\label{160303}
$f'''_-(u+\eps)\ge0.$
\end{enumerate}
\end{Th}

\begin{Th}
\label{160312}
For the point $(u-\eps,\,\eps)$ to be a tip of a right herringbone in the strip $\Oe,$ the following conditions must be satisfied\textup:
\begin{enumerate}
\item
\label{160304}
$D_+(u,-\eps)=0;$
\item
\label{160305}
$f''_+(u-\eps)-f''_-(u+\eps)+2\eps f'''_+(u-\eps)\ge0;$
\item
\label{160306}
$f'''_+(u-\eps)\le0.$
\end{enumerate}
\end{Th}

\begin{Th}
\label{160313}
For the point $(u-\eps,\,-\eps)$ to be a tip of a right herringbone in the strip $\Oe,$ the following conditions must be satisfied\textup:
\begin{enumerate}
\item
\label{160307}
$D_-(u,\eps)=0;$
\item
\label{160308}
$f''_+(u+\eps)-f''_-(u-\eps)-2\eps f'''_-(u-\eps)\le0;$
\item
\label{160309}
$f'''_-(u-\eps)\le0.$
\end{enumerate}
\end{Th}

Let us now turn to analogous sufficient conditions for the construction of a horizontal herringbone with a tip on the boundary. It turns out that if, in conditions~\eqref{140302} and~\eqref{140303} of the preceding statements, we replace the non-strict inequalities by strict ones, we obtain sufficient conditions for the possibility of constructing the corresponding horizontal herringbones.

\begin{Th}
\label{170301}
To be able to construct a left herringbone with the tip at the point $(u+\eps,\,\eps)$ in the strip $\Oe,$ it is sufficient that the following conditions be satisfied\textup:
\begin{enumerate}
\item
\label{170302}
$D_+(u,\eps)=0;$
\item
\label{170303}
$f''_+(u+\eps)-f''_-(u-\eps)-2\eps f'''_+(u+\eps)>0;$
\item
\label{170304}
$f'''_+(u+\eps)>0.$
\end{enumerate}
\end{Th}

\begin{proof}
Our task reduces to showing the existence of an integral curve $\ell_\eps$ of the field from~(5.3) that enters the strip $\Oe$ through the point $(u,\eps)$ and, in some neighbourhood of this point, passes through the domain where the inequalities $D_+>0$ and $D_->0$ hold. Most of the technical material has already been prepared in the proof of Theorem~\ref{160310}. Thus, we know that the curve $\ell_\eps$ leaves the point $(u,\eps)$ with the slope~\eqref{150304}, which is greater than the slope~\eqref{150305} of the curve $D_+=0$ whenever $\kappa_+^\LL>0$, as guaranteed by conditions~\eqref{170303} and~\eqref{170304}. Consequently, the curve $\ell_\eps$ lies to the right of the curve $D_+=0$. To the right of this curve lies the domain where $D_+>0$, which also follows from condition~\eqref{170303}. 
The inequality $D_->0$ is ensured by the combination of conditions~\eqref{170303} and~\eqref{170304}. Indeed,
$$
\begin{aligned}
D_-(u,\eps)&=D_+(u,\eps)+\big[f''_+(u+\eps)-f''_-(u-\eps)\big]=
\\
&=\big[f''_+(u+\eps)-f''_-(u-\eps)-2\eps f'''_+(u+\eps)\big]+2\eps f'''_+(u+\eps)>0\,.
\end{aligned}
$$
By continuity, the positivity of the function $D_-$ is preserved in some neigh\-bor\-hood of the point $(u,\eps)$. This argument completes the proof.
\end{proof}

The statements of the symmetric propositions are obvious.

\begin{Th}
\label{170305}
To be able to construct a left herringbone with the tip at the point $(u+\eps,\,-\eps)$ in the strip $\Oe,$ it is sufficient that the following conditions be satisfied\textup:
\begin{enumerate}
\item
\label{170306}
$D_-(u,-\eps)=0;$
\item
\label{170307}
$f''_+(u-\eps)-f''_-(u+\eps)+2\eps f'''_-(u+\eps)<0;$
\item
\label{170308}
$f'''_-(u+\eps)>0.$
\end{enumerate}
\end{Th}

\begin{Th}
\label{170309}
To be able to construct a right herringbone with the tip at the point $(u-\eps,\,\eps)$ in the strip $\Oe,$ it is sufficient that the following conditions be satisfied\textup:
\begin{enumerate}
\item
\label{170310}
$D_+(u,-\eps)=0;$
\item
\label{170311}
$f''_+(u-\eps)-f''_-(u+\eps)+2\eps f'''_+(u-\eps)>0;$
\item
\label{170312}
$f'''_+(u-\eps)<0.$
\end{enumerate}
\end{Th}

\begin{Th}
\label{170313}
To be able to construct a right herringbone with the tip at the point $(u-\eps,\,-\eps)$ in the strip $\Oe,$ it is sufficient that the following conditions be satisfied\textup:
\begin{enumerate}
\item
\label{170314}
$D_-(u,\eps)=0;$
\item
\label{170315}
$f''_+(u+\eps)-f''_-(u-\eps)-2\eps f'''_-(u-\eps)<0;$
\item
\label{170316}
$f'''_-(u-\eps)<0.$
\end{enumerate}
\end{Th}

\section{Major pockets}
\label{141002}

Our immediate task is to describe the evolution of two unidirectional fissures. For definiteness, we shall consider left fissures and prove the following theorem.

\begin{Th}
\label{160201}
Suppose that there exists a value $\eps_0$ of the parameter $\eps$ such that\textup:
\begin{itemize}
\item
in some neighborhood of $\eps_0$ there exists a function 
$u=u_+(\eps)$ satisfying the conditions of 
Theorem~\textup{\ref{170301}} and\textup, for 
$\eps\le\eps_0$\textup, generating an \textup{SW}-fissure 
on the domain $\Om{SW}(v_-,u_+)$\textup;
\item
in some neighborhood of $\eps_0$ there exists a function 
$u=u_-(\eps)$ satisfying the conditions of 
Theorem~\textup{\ref{170305}}\textup, generating an \textup{NW}-fissure 
on the domain $\Om{NW}(v_+,u_-)$\textup;
\item
for $\eps<\eps_0$ the first herringbone lies to the right of the second \textup{\bf(}i.\,e.\textup, 
$u_-(\eps)<v_-(\eps)$\textup{\bf)}\textup, between them 
lies a domain with a simple left foliation 
$\Om{L}(u_-,v_-),$ and the value $\eps_0$ is determined by the equation $v_-(\eps_0)=u_-(\eps_0).$
\end{itemize}
Then\textup, in some right neighborhood of $\eps_0$\textup, 
the two described herringbones ``glue'' together in the following 
sense\textup: a new left herringbone appears\textup, whose tip is located at the point 
$(u_+(\eps)+\eps,\eps)$\textup, and as $\eps\searrow\eps_0$ the spine function $T(u)$ of this new herringbone 
tends to the spine function of the first herringbone for $u\ge u_-(\eps_0)+\eps_0,$ and to the spine function of 
the second herringbone for $u\le u_-(\eps_0)+\eps_0$.
\end{Th}

\begin{proof} 
First, we note that the possibility of constructing the fissures mentioned in the statement implies the existence of points with coordinates $u_{01}$
and $u_{02}$ on the axis $x_2=0$ at which the integral
lines of our field intersect this axis. 
See Figure~\ref{190801}. (The fissures will intersect the 
axis $x_2=0$ at the points 
$u_{01}+\eps$ and $u_{02}+\eps$, respectively.)

\begin{figure}[h]
    \centering
    \includegraphics[scale = 0.25]{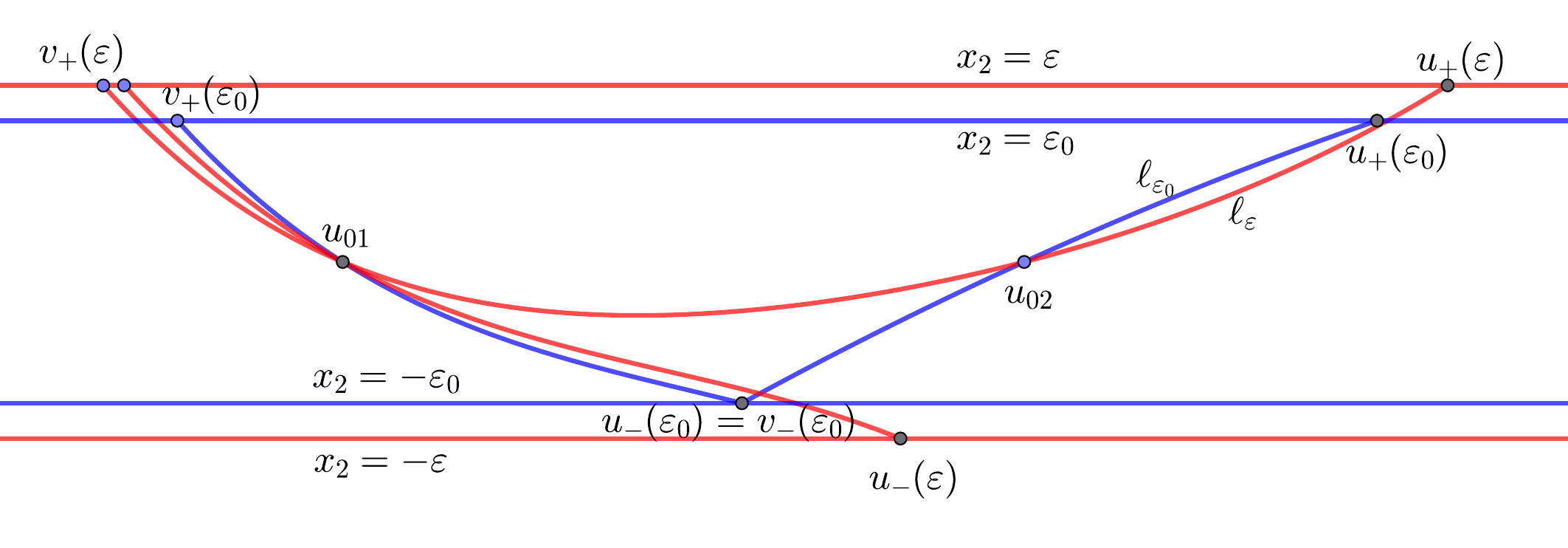}
    \caption{Integral curves for $\eps_0$ and 
    for $\eps>\eps_0$.}
    \label{190801}
\end{figure}

Thus, for $\eps\le\eps_0$ our local foliation
has the following form:\footnote{Recall that
in Section~11 we introduced for domains foliated by 
horizontal herringbones the notation $\Om{HB} 
(A,C)$ if the spine of the herringbone starts at the point $A$ 
and ends at the point $C$. In the present case, $\Om{SW}(v_-,u_+)=\Om{HB}\big((v_-+\eps,\eps),(u_++\eps,\eps)\big)$ and 
$\Om{NW}(v_+,u_-)=\Om{HB}\big((v_++\eps,\eps),(u_-+\eps,\eps)\big)$.}

\eq{160701}{
\begin{gathered}
 \Om{R}(\ast,v_+)\cup\Om{HB}\big((v_++\eps,\eps),(u_-+\eps,\eps)\big)\cup\Om{L}(u_-,v_-)\cup
 \\
 \cup\Om{HB}\big((v_-+\eps,\eps),(u_++\eps,\eps)\big)\cup\Om{R}(u_+,\ast)\,. 
\end{gathered}
}

For $\eps=\eps_0$ the domain $\Om{L}(u_-,v_-)$ 
collapses to a segment, and such a foliation
can no longer exist. The assertion of the theorem being proved 
is that for $\eps>\eps_0$ 
the local foliation takes the form
\eq{160702}{
\Om{R}(\ast,v_+)\cup\Om{HB}\big((v_++\eps,\eps),(u_++\eps,\eps)\big)\cup\Om{R}(u_+,\ast)\,.
}

Note that, by the hypotheses of the theorem, for $\eps>\eps_0$
we could continue to construct a fissure from the point
$(u_-+\eps,-\eps)$, but we would not be able to connect it in any way
with the foliation coming from the right.
The herringbone whose tip is located at the point
$(u_++\eps,\eps)$ for $\eps>\eps_0$ will no
longer generate a fissure. It will be, so to speak, the 
``heir'' of both fissures.

Let us examine in more detail the integral curve 
$\ell_\eps$ that enters the strip $\Oe$ with 
a positive slope at the saddle point 
$(u_+(\eps),\eps)$. If $\eps<\eps_0$, then by 
Theorem~\ref{170301} the curve $\ell_\eps$ generates 
a herringbone which, by the hypotheses of the theorem, is 
a fissure. Now let us increase $\eps$ slightly.
Then, according to Proposition~8.8, as $\eps$ grows the curves will 
lie below the curve 
$\ell_{\eps_0}$ in the upper half of the strip,
descending toward the axis $x_2=0$ with increasing $\eps$. 
For all $\eps$ the curves $\ell_\eps$ intersect 
the midline of the strip at the point $u_{02}$. After 
crossing the midline, the curves $\ell_\eps$ are 
ordered in the opposite way: the larger $\eps$, the higher the curve 
lies (see Remark~8.12 and the illustration in 
Figure~\ref{190801}). 
Thus, for $\eps>\eps_0$ the curve 
$\ell_\eps$ passes above the point 
$(u_-(\eps_0),-\eps_0)$. Moreover, if we 
consider the curve $\ell_\eps^-$ issuing from 
the saddle point $(u_-(\eps),-\eps)$ (this is the curve 
that would generate a fissure, but we now do not construct it), 
then we see that it separates the curve 
$\ell_\eps$ from the lower boundary of the strip on the interval
$x_1\in(u_{01},u_-(\eps_0))$. At the node $(u_{01},0)$
these curves intersect, and in the upper half of 
the strip the curve $\ell_\eps^-$ lies already 
above the curve $\ell_\eps$ (see Fig.~\ref{190801}).

The fact that the curve $\ell_\eps$ is close to 
the curves $\ell_{\eps_0}$ and $\ell^-_{\eps_0}$ when $\eps$ is close to $\eps_0$ follows from 
the sufficient smoothness of the boundary values.
\end{proof}

\section{Examples. Third-degree polynomials: 
merger of two fissures}
\label{141003}

We shall apply the theorem just proved to the situation 
described in Fig.~14 (see~\cite{First}). Recall that there 
we considered the case of third-degree polynomials $f_\pm$ 
with $a^\pm_3>0$, and the discriminant of the quadratic polynomial 
$f'_+-f'_-$ is positive, i.\,e.,
\eq{100704}{
d^2\df(a_2^+-a_2^-)^2-3(a_3^+-a_3^-)(a_1^+-a_1^-)>0\,.
}
Therefore, the equation $f'_+(u)=f'_-(u)$ has two distinct real roots $u_{01}$ and $u_{02}$. 
We then found that for small $\eps$ the foliation has the form
\eq{100701}{
\begin{aligned}
\Oe=\Om{R}(-\infty,v_{+1})\cup\Om{NW}(v_{+1},u_{-1})\cup\Om{L}(u_{-1},v_{-2})\cup
\\
\cup\Om{SW}(v_{-2},u_{+2})\cup\Om{R}(u_{+2},+\infty)\,,
\end{aligned}
}
and the domain $\Om{L}(u_{-1},v_{-2})$ shrinks as $\eps$ grows.

Recall that $u_{+2}$ is the larger root of the equation $D_+(u,\eps)=0$, i.\,e.,
\eq{141101}{
u_{+2}=-\eps+\frac{-(a_2^+-a_2^-)+\sqrt{d^2+36(a_3^+-a_3^-)a_3^+\eps^2}}{3(a_3^+-a_3^-)}\,,
}
and $u_{-1}$ is the smaller root of the equation $D_-(u,-\eps)=0$, i.\,e.,
\eq{100702}{
u_{-1}=\eps+\frac{-(a_2^+-a_2^-)-\sqrt{d^2-36(a_3^+-a_3^-)a_3^-\eps^2}}{3(a_3^+-a_3^-)}\,.
}
The values $u_{+2}$ and $u_{-1}$ satisfy the conditions of Theorem~\ref{170301} and Theorem~\ref{170305}, respectively. According to these theorems, we can construct left herringbones with tips at the points $(u_{+2}+\eps,\eps)$ and $(u_{-1}+\eps,-\eps)$. For small $\eps$, these herringbones extend to the opposite boundary, and their bases are denoted by $(v_{-2}+\eps,-\eps)$ and $(v_{+1}+\eps,\eps)$. To find the values $v_{-2}$ and $v_{+1}$, one needs to solve the differential equation~(4.32) with boundary conditions at $u=u_{+2}+\eps$ and $u=u_{-1}+\eps$, respectively. We cannot solve this equation explicitly, so we must resort to indirect considerations.

We are not interested in the value $v_{+1}$; it suffices for us to know the simple fact that to the left of the point $(v_{+1}+\eps,\eps)$ we can construct a right simple foliation for all $\eps$. 
Much more information is needed about $v_{-2}$. For constructing the herringbone, it is important that the corresponding integral curve $\ell_\eps$ passes through the domain where $D_+\ge0$ and $D_-\ge0$. Thus, the value $v_{-2}(\eps)$ must lie between the roots of the quadratic polynomials $D_+(u,-\eps)$ and $D_-(u,-\eps)$. While the first condition will cause us no difficulties, the second simply cannot hold for all $\eps$. The curve $D_-=0$ is an ellipse, and the interval $(u_{-1},u_{-2})$ between the roots of this polynomial shrinks to a point at
\eq{201203}{
\eps_{\max}=\frac{d}{6\sqrt{a_3^-(a_3^+-a_3^-)}}\,.
}
Thus, the value $v_{-2}(\eps)$, which is a continuous function of $\eps$, must for some $\eps$ ($\eps\le\eps_{\max}$) coincide with either the left or the right endpoint of the interval $(u_{-1},u_{-2})$. Our task is to show that for all admissible values of the coefficients of the boundary polynomials this will be the left endpoint, i.\,e., we shall prove that there exists $\eps_0$ ($\eps_0<\eps_{\max}$) solving the equation $v_{-2}(\eps)=u_{-1}(\eps)$, and for all $\eps<\eps_0$ the strict inequality $u_{-1}<v_{-2}<u_{-2}$ holds. After that, we will be able to apply Theorem~\ref{160201} and conclude that a major pocket is formed from the two fissures, and for $\eps>\eps_0$ the foliation takes the form
\eq{141102}{
\Oe=\Om{R}(-\infty,v_{+1})\cup\Om{HB}\big((v_{+1}+\eps,\eps),(u_{+2}+\eps,\eps)\big)\cup\Om{R}(u_{+2},+\infty)\,.
}

Until now, we have only needed the estimate $v_{-2}<u_{02}-\eps$, which followed from the fact that the positive slope of the spine (the curve $\ell_\eps$) is not grater than one. We now show that this will suffice only under the condition
\eq{191101}{a_3^+\ge2a_3^-.}

Indeed, if we try to verify that for $\eps\le\eps_{\max}$ the condition 
$u_{02}-\eps\le u_{-2}(\eps)$ holds \textup{\bf(}which guarantees 
the inequality $v_{-2}(\eps)<u_{-1}(\eps)$\textup{\bf)},
then we need the inequality
$$
d-6(a_3^+-a_3^-)\eps\le\sqrt{d^2-36(a_3^+-a_3^-)a_3^-\eps^2}
$$
to hold for all $\eps\le\eps_{\max}$. In particular, for $\eps=\eps_{\max}$
we obtain
$$
d\le6(a_3^+-a_3^-)\eps_{\max}=d\sqrt{\frac{a_3^+-a_3^-}{a_3^-}},
$$
which is equivalent to condition~\eqref{191101}. For the remaining values of the coefficients (i.\,e., when $a_3^-\le a_3^+<2a_3^-$) we will have to find other arguments.

We shall examine in more detail the behavior of the integral curves of our field in order to verify that the point $(v_{-2}(\eps),-\eps)$ can never lie to the right of the point $(u_{-2}(\eps),-\eps)$.

In the lower half-plane we may have three stationary points — these are the intersection points of the curve $D_-=0$ with the boundary $x_2=-\eps$ and the point $x_\ast$, independent of $\eps$, where the curves $D_+=0$ and $D_-=0$ intersect. For the values of the coefficients of the boundary polynomials under consideration ($0<a_3^-\le a_3^+$), the curve $D_+=0$ consists of two branches of a hyperbola (together with the boundary $x_2=-\eps$ they form the isocline $X_1$, at whose points the integral curves have slope $1$), while the curve $D_-=0$ is an ellipse (together with the boundary $x_2=\eps$ it forms the isocline $X_{-1}$, at whose points the integral curves have slope $-1$). Their relative position is shown in Figure~\ref{201101}.
\begin{figure}[h]
    \centering
    \includegraphics[scale = 0.4]{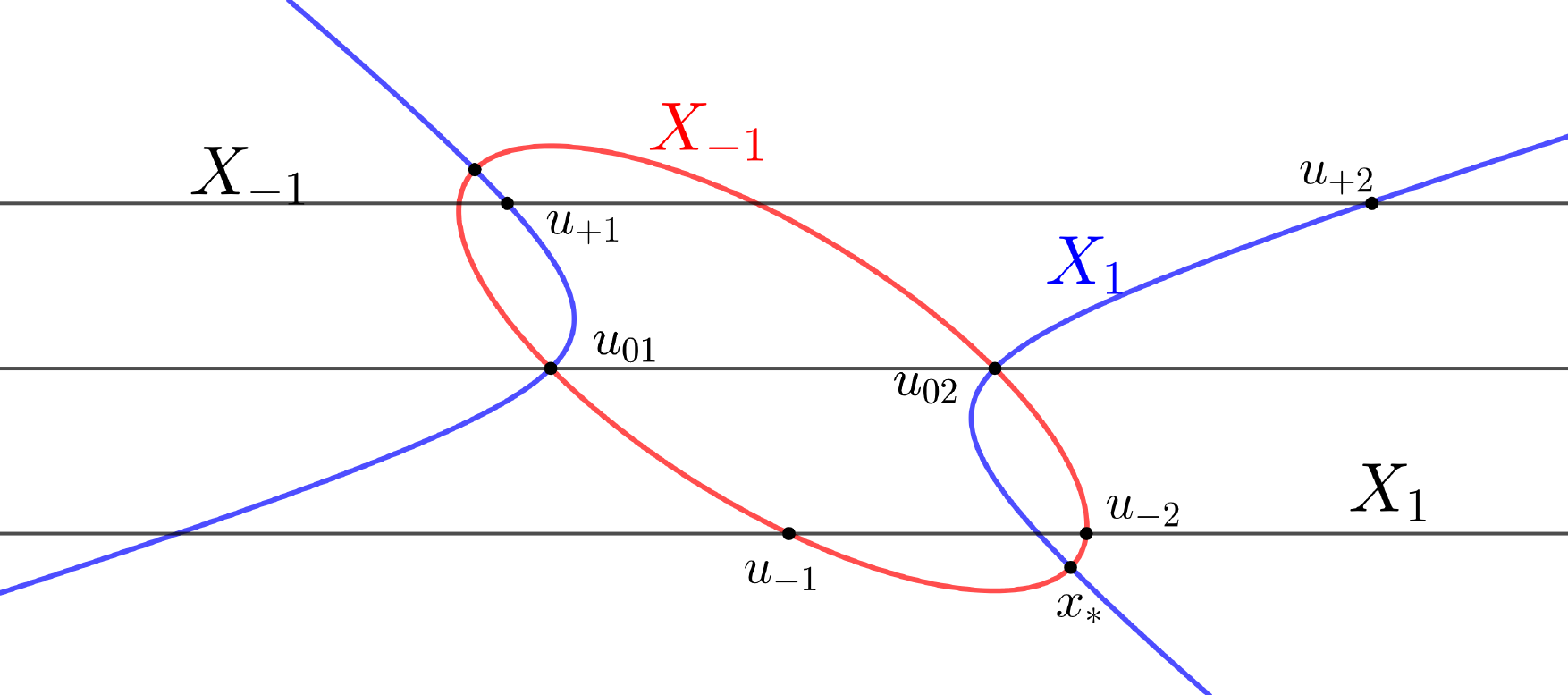}
    \caption{The curves $D_+=0$ and $D_-=0$.}
    \label{201101}
\end{figure}
In the lower half-plane, the solution of the system $D_+=D_-=0$ has the form
\eq{201102}{
x_\ast=(x_{\ast1},x_{\ast2})=
\Big(u_0+\frac{a_3^++a_3^-}{a_3^+-a_3^-}\,\eps_\ast,
\;-\eps_\ast\Big),
}
where
\eq{080201}{
\eps_\ast\df\frac{d}{6\sqrt{a_3^+a_3^-}}\,.
}
Recall that
\eq{201103}{
u_0=-\frac{a_2^+-a_2^-}{3(a_3^+-a_3^-)}\,.
}

For $\eps\le\eps_\ast$ we have no problems, since, as we know, the curve $\ell_\eps$ can intersect the curve $D_+=0$ only at stationary points (and here this is moreover obvious: in the lower half of the strip the curves $\ell_\eps$ and $D_+=0$ diverge). Therefore, the right branch of the hyperbola $D_+=0$ separates the curve $\ell_\eps$ in the lower half of the strip from the right part of the boundary of the ellipse $D_-=0$, i.e., $\ell_\eps$ cannot leave the strip through the point $(u_{-2},-\eps)$. Consequently, we need to understand the behavior of the integral curves of the field from~(5.3) under investigation for $\eps\in(\eps_\ast,\eps_{\max})$.

The Jacobi matrix for our field was written out in general form in \cite{First} \textup{\bf(}formula (5.5)\textup{\bf)}. Moreover, the form of the Jacobi matrix at the points $u_-$, where the curve $D_-=0$ intersects the boundary $x_2=-\eps$, was also given there:
\eq{211101}{
J(u_-,-\eps)=\eps\big(f_+''(u_--\eps)-f_-''(u_-+\eps)+2\eps f_-'''(u_-+\eps)\big)
\begin{pmatrix}
1&2+3\eps\kappa_-
\\
1&\eps\kappa_-
\end{pmatrix},
}
where
\eq{211102}{
\kappa_-=\frac{2f_-'''(u_-+\eps)}{f_-''(u_-+\eps)-f_+''(u_--\eps)-2\eps f_-'''(u_-+\eps)}\,.
}
In our case of third-degree polynomials, we obtain the following expression:
\eq{221101}{
\kappa_-=-\frac{2a_3^-}{(a_3^+-a_3^-)(u_--\eps-u_0)}=
\pm\frac{6a_3^-}{\sqrt{d^2-36(a_3^+-a_3^-)a_3^-\eps^2}}\,.
}
The plus sign in the last formula corresponds to $\kappa_{-1}$, i.\,e., to the root $u_{-1}$. Accordingly, the minus sign corresponds to $\kappa_{-2}$, i.\,e., to the root $u_{-2}$.

If we introduce the notation $s=1+\eps\kappa_-$, then the matrix in formula~\eqref{211101} takes the form
\eq{221102}{
\begin{pmatrix}
1&3s-1
\\
1&s-1
\end{pmatrix},
}
The characteristic polynomial of this matrix is $\lambda^2-s\lambda-2s$,
the eigenvalues are $\half(s\pm\sqrt{s^2+8s})$, and the eigenvectors of this matrix have the form
\eq{101201}{
\begin{pmatrix}
\lambda-s+1\\1
\end{pmatrix}. 
}
Since $\kappa_{-1}>0$, the matrix~\eqref{221102} has two real eigenvalues of opposite signs for all $\eps<\eps_{\max}$. Consequently, the left intersection point of the curve $D_-=0$ with the boundary $x_2=-\eps$, i.\,e., the point $(u_{-1},-\eps)$, is always a saddle point. It satisfies Theorem~\ref{170305}, and for all $\eps<\eps_{\max}$ we shall construct a fissure, namely, a left horizontal herringbone whose tip is located at the point $(u_{-1}+\eps,-\eps)$.

The nature of the stationary point $(u_{-2},-\eps)$ changes as $\eps$ increases. If
\eq{080202}{
\eps_{\ast}<\eps<\eps_1\df\frac{3d}{2\sqrt{(81a_3^+-80a_3^-)a_3^-}}\,,
}
then $-8<s<0$, and the characteristic polynomial of the Jacobi matrix has complex conjugate roots. Thus the point $(u_{-2},-\eps)$ is a focus. What matters to us is that no integral curve can leave the strip through this point.

If
\eq{080203}{
\eps_1<\eps<\eps_{\max}\,,
}
then $s<-8$, and the Jacobi matrix has two distinct negative eigenvalues. Hence the point $(u_{-2},-\eps)$ is a nondegenerate node.

Let us now examine the Jacobi matrix at the point $x_{\ast}$:
\eq{281101}{
J(x_\ast)=-6\eps_{\ast}
\begin{pmatrix}
-(\eps+\eps_{\ast})a_3^--(\eps-\eps_{\ast})a_3^+&(\eps+\eps_{\ast})a_3^--(\eps-\eps_{\ast})a_3^+
\\
-(\eps+\eps_{\ast})a_3^-+(\eps-\eps_{\ast})a_3^+&(\eps+\eps_{\ast})a_3^-+(\eps-\eps_{\ast})a_3^+
\end{pmatrix}.
}
The eigenvalues of this matrix are
\eq{281102}{
\lambda=\pm12\eps_{\ast}\sqrt{(\eps^2-\eps_{\ast}^2)a_3^+a_3^-}.
}
For $\eps>\eps_{\ast}$, we have two real eigenvalues of opposite signs. Hence $x_\ast$ is a saddle point. The eigenvectors of this matrix can be written as
\eq{131201}{
\begin{pmatrix}
\sqrt{(\eps+\eps_{\ast})a_3^-}+\sqrt{(\eps-\eps_{\ast})a_3^+}
\\
\sqrt{(\eps+\eps_{\ast})a_3^-}-\sqrt{(\eps-\eps_{\ast})a_3^+}
\end{pmatrix}
\text{ and }
\begin{pmatrix}
\sqrt{(\eps+\eps_{\ast})a_3^-}-\sqrt{(\eps-\eps_{\ast})a_3^+}
\\
\sqrt{(\eps+\eps_{\ast})a_3^-}+\sqrt{(\eps-\eps_{\ast})a_3^+}
\end{pmatrix}.
}
Since $(\eps+\eps_{\ast})a_3^- > (\eps-\eps_{\ast})a_3^+$ for $\eps<\eps_{\max}$, the slopes of these vectors are positive and complementary. The curve $D_-=0$, which has slope $1$ at $x_{\ast}$, lies in the narrow angle between these vectors when $\eps$ slightly exceeds $\eps_{\ast}$. Therefore, the point $(u_{-2},-\eps)$, being the intersection point of the ellipse $D_-=0$ with the boundary of the strip, lies between the integral curves emerging from $x_{\ast}$ along the vectors~\eqref{131201}. What matters to us is that the separatrix (call it $L$) emerging from $x_{\ast}$ along the left vector from~\eqref{131201} (the vector with the smaller slope) cannot intersect the boundary of the strip to the right of the point $(u_{-2},-\eps)$. This holds not only for $\eps$ slightly exceeding $\eps_{\ast}$ but for all $\eps\le\eps_{\max}$. To see this, we consider the isocline $X_0$, defined as the locus of points at which the integral curves of our field are horizontal, i.\,e.,
\eq{191201}{
\dot x_2=x_2(f_-'-f_+')-x_2(\eps-x_2)f_-''+x_2(\eps+x_2)f_+''=0
}
\textup{\bf(}see formula (5.3)\textup{\bf)}. The horizontal axis $x_2=0$ is not of interest to us at the moment. Rather, we are interested in the branch of the hyperbola that passes through $x_{\ast}$; more specifically, the part connecting $x_{\ast}$ and $(u_{-2},-\eps)$. The relative position of the curves $X_0$, $X_{-1}$ (points where the slope of the integral curves is $-1$, i.\,e., where $D_-=0$), and $X_\infty$ (points where the integral curves are vertical, i.\,e., $\dot x_1=0$) is shown in Figure~\ref{191202}.
\begin{figure}[h]
    \centering
    \includegraphics[scale = 0.35]{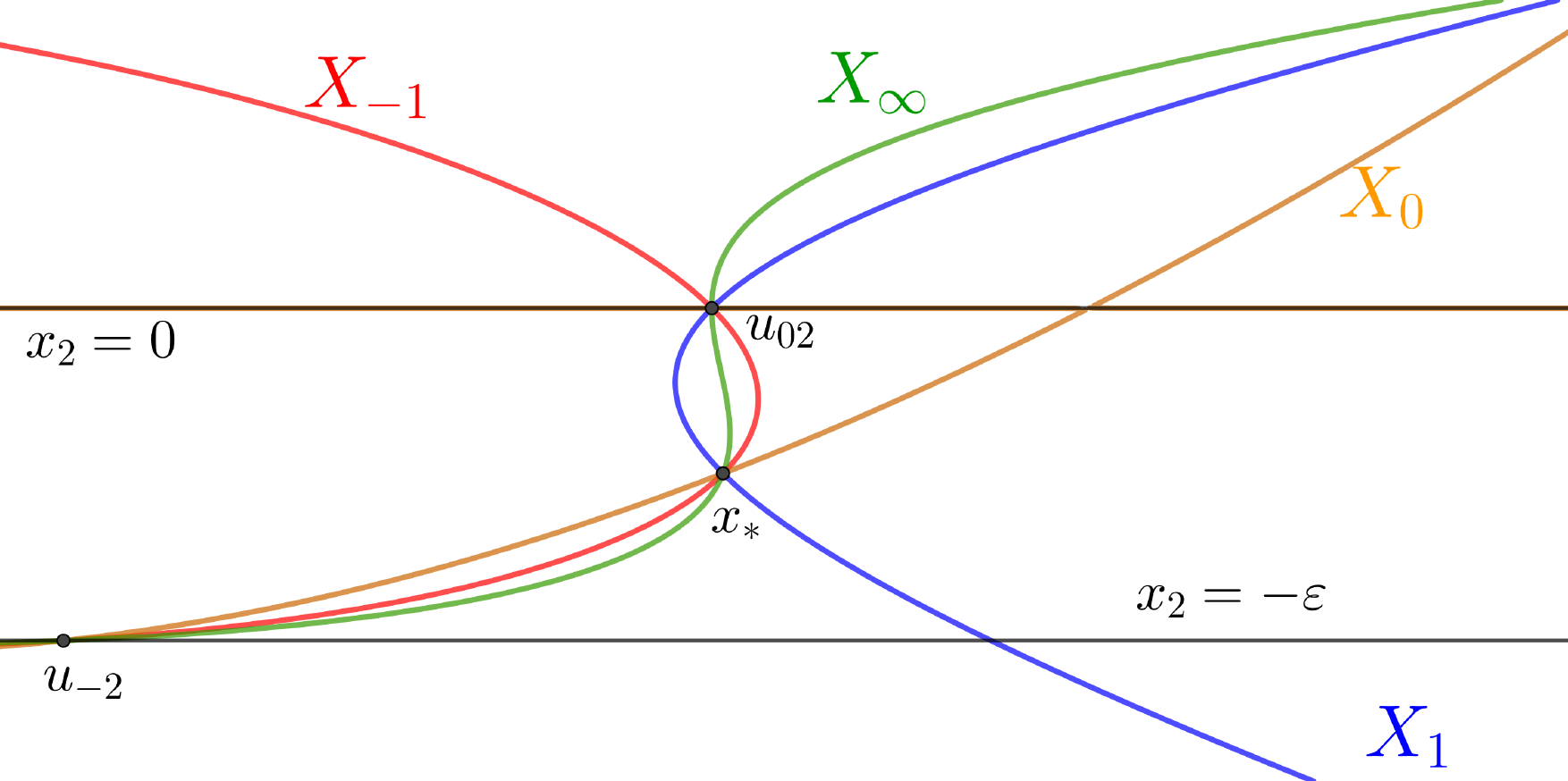}
    \caption{The isoclines $X_0$, $X_{-1}$, and $X_\infty$ between the points
    $x_{\ast}$ and $(u_{-2},-\eps)$.}
    \label{191202}
\end{figure}
Such a position of these curves could be explained qualitatively by the fact that the integral curves must turn sharply in a neighborhood of the saddle point $x_{\ast}$. However, let us verify this analytically, as well as the fact that all these curves lie between the two separatrices of our field emerging from $x_{\ast}$ along the vectors~\eqref{131201}.

Rewriting equation~\eqref{191201} in terms of the coefficients of the polynomials $f_+$ and $f_-$, we obtain
\eq{191203}{
\begin{aligned}
3(a_3^+-a_3^-)(x_2^2-x_1^2)&+6\eps[a_3^+(x_1+x_2)-a_3^-(x_1-x_2)]
\\
&+2(a_2^+-a_2^-)(\eps-x_1)-(a_1^+-a_1^-)=0\,.   
\end{aligned}
}
Differentiating, we obtain an expression for the slope of the curve $X_0$:
\eq{191204}{
\frac{x_1-\eps-u_0}{x_2+\frac{a_3^++a_3^-}{a_3^+-a_3^-}\eps}\,,
}
where $u_0$, recall, is given by formula~\eqref{201103}. Substituting the coordinates of the point $x_{\ast}$ \textup{\bf(}see~\eqref{131201}\textup{\bf)} here, we obtain the expression
\eq{191205}{
\frac{a_3^-(\eps+\eps_{\ast})-a_3^+(\eps-\eps_{\ast})}
{a_3^-(\eps+\eps_{\ast})+a_3^+(\eps-\eps_{\ast})}\,.
}

As we have already noted, this slope is positive, i.\,e.,
\eq{201201}{
a_3^-(\eps+\eps_{\ast})>a_3^+(\eps-\eps_{\ast})
}
for all $\eps<\eps_{\max}$. Indeed, by rewriting condition~\eqref{201201} in the form
\eq{201202}{
\eps<-\frac{a_3^++a_3^-}{a_3^+-a_3^-}\,x_{\ast2}\,,
}
it is easy to see that this bound is larger than $\eps_{\max}$. Indeed,
\begin{align*}
\eps_{\max}\buildrel{\eqref{201203}}\over=&
\frac{d}{6\sqrt{a_3^-(a_3^+-a_3^-)}}\buildrel{\eqref{080201}}\over=
\eps_{\ast}\;\sqrt{\frac{a_3^+}{a_3^+-a_3^-}}
\\
<&\ \eps_{\ast}\;\frac{a_3^+}{a_3^+-a_3^-}
< \eps_{\ast}\;\frac{a_3^++a_3^-}{a_3^+-a_3^-}\,.
\end{align*}

Note that the positive slope of the first vector ftom~\eqref{131201} is less than the slope of the curve $X_0$ at $x_{\ast}$:
$$
\begin{gathered}
\frac{\sqrt{(\eps+\eps_{\ast})a_3^-}-\sqrt{(\eps-\eps_{\ast})a_3^+}}
{\sqrt{(\eps+\eps_{\ast})a_3^-}+\sqrt{(\eps-\eps_{\ast})a_3^+}}=
\frac{(\eps+\eps_{\ast})a_3^--(\eps-\eps_{\ast})a_3^+}
{\Big(\sqrt{(\eps+\eps_{\ast})a_3^-}+\sqrt{(\eps-\eps_{\ast})a_3^+}\Big)^2}
\\
<\frac{(\eps+\eps_{\ast})a_3^--(\eps-\eps_{\ast})a_3^+}
{(\eps+\eps_{\ast})a_3^-+(\eps-\eps_{\ast})a_3^+}\,.
\end{gathered}
$$

Thus, the integral curve $L$ emerges at a smaller positive angle than the isocline $X_0$. Consequently, below the point $x_{\ast}$ it lies above $X_0$ (to the left of $X_0$). This relative position continues throughout the interval $x_1\in(u_{-2},x_{\ast1})$, since on this interval, where there are no singular points, the curve cannot intersect $X_0$. What matters to us is that this integral curve cannot intersect the boundary of the domain to the right of the point $(u_{-2},-\eps)$.

Now recall that we needed the same assertion for another integral curve of our field, which we denoted by $\ell_\eps$. Let us return to it. We need to verify that if this curve reaches the lower boundary of the strip (we shall assume this from now on), then its intersection point with this boundary cannot lie to the right of the point $(u_{-2},-\eps)$.

As already noted, the curve $\ell_\eps$ has two common points with the curve $X_1$ (i.\,e., $D_+=0$): the point on the upper boundary from which we start this line, and the intersection point with the midline of the strip. In the upper half of the strip it lies to the right of the hyperbola $X_1$, and in the lower half — to the left (more precisely, between the branches of the hyperbola, i.\,e., to the left of the right branch). Thus, the curve $\ell_\eps$ intersects the line $x_2=-\eps_{\ast}$ to the left of the point $x_{\ast}$. Consequently, on its entire path to the lower boundary it lies to the left of the integral curve $L$, which is what we needed.

In fact, we wanted slightly more: we wanted the strict inequality $\eps_0<\eps_{\max}$ to hold. Although I believe that the curve $\ell_\eps$ never intersects the lower boundary at the point $(u_{-2},-\eps)$ for any values of the parameters, even if this does occur, nothing prevents the construction of the claimed foliation. To apply Theorem~\ref{160201}, we need the strict inequality~\eqref{170307}, which fails only in the case $\eps_0=\eps_{\max}$, i.\,e., if $v(\eps_0)=u_{-1}(\eps_0)=u_{-2}(\eps_0)=u_0+\eps_0$.

At $\eps=\eps_{\max}$, a so-called saddle-node bifurcation occurs: the two stationary points — the saddle $(u_{-1}(\eps),-\eps)$ and the node $(u_{-2}(\eps),-\eps)$ — merge into one, $(u_0+\eps_{\max},-\eps_{\max})$, which is called a saddle-node, and then disappear.

At the saddle-node point $(u_0+\eps_{\max},-\eps_{\max})$, the Jacobi matrix equals
\eq{200301}{
J(u_0+\eps_{\max},-\eps_{\max})=-2\eps_{\max}^2f'''_-(u_0)
\begin{pmatrix}
    0&3\\0&1
\end{pmatrix}
=-12\eps_{\max}^2a_3^-
\begin{pmatrix}
    0&3\\0&1
\end{pmatrix}.
}
It has a horizontal eigenvector corresponding to the zero eigenvalue, while the eigenvector with slope $\tfrac13$ corresponds to the second eigenvalue. We are interested in the separatrix running along the horizontal eigenvector. In both directions it passes inside the strip. To the right, however, it lies above the isocline $X_0$, while to the left it lies below $X_0$ but above the isocline $X_{-1}$, i.\,e., in the domain where $D_->0$. Therefore, although we cannot apply Theorem~\ref{170305} because the strict condition~\eqref{170307} is not satisfied, we can nevertheless construct the corresponding herringbone, whose tip touches the lower boundary. After that, we can use the arguments from the proof of Theorem~\ref{160201} to construct a major pocket for $\eps>\eps_{\max}$.

Until now, we have followed the curve $\ell_\eps$ up to the moment $\eps=\eps_0$,  when the two fissures merge. 
Now, in order to guarantee the possibility of constructing the herringbone, we need to verify that after this moment the curve $\ell_\eps$ passes through the domain where $D_+>0$ and $D_->0$ for all $\eps$. 
Moreover, for smaller $\eps$ we also need to verify that the second fissure passes through this domain. Let us start with that fissure (we shall denote the corresponding integral curve by $\ell_\eps^-$). 
The same argument that showed that the curve $\ell_\eps$, starting from a point on the upper boundary where $D_+=0$, 
later in the upper half-strip, up to the stationary point $(u_{02},0)$, lies below the isocline $X_1$ ($D_+=0$), 
i.\,e., in the domain where $D_+>0$, also shows that the curve $\ell_\eps^-$, 
starting from a point on the lower boundary where $D_-=0$, later in the lower half-strip, up to the stationary point $(u_{01},0)$, lies above the isocline $X_{-1}$ ($D_-=0$), 
i.\,e., in the domain where $D_->0$. 
Accordingly, from the stationary point $(u_{01},0)$ the curve $\ell_\eps^-$ emerges below the isocline $X_{-1}$ (or, in other words, to the left of it) and diverges from it, moving to the left, thus remaining in the domain where $D_->0$. 
The fact that the curve $\ell_\eps^-$ passes between the branches of the hyperbola $X_1$ in the lower half-plane (i.\,e., in the domain where $D_+>0$) is obvious. 
The same holds for $\eps>\eps_0$, i.\,e., after the two fissures merge, when only one curve $\ell_\eps$ remains, 
which, without reaching the lower boundary, intersects the isocline $X_0$ and turns upward. 
Thus, its segment between the stationary points $(u_{02},0)$ and $(u_{01},0)$ lies in the lower half-strip in the domain where $D_+>0$ and $D_->0$.

Now let us follow the behavior of the integral curve of interest in the upper half of the strip to the left of the point $(u_{01},0)$. Here it does not matter whether $\eps>\eps_0$ or $\eps<\eps_0$, i.\,e., whether this is an arc of the curve $\ell_\eps^-$ or already an arc of $\ell_\eps$. What matters to us is that it emerges from the point $(u_{01},0)$ with a negative slope strictly greater than $-1$, while the ellipse $X_{-1}$ has slope exactly $-1$ at this point, and further in the upper strip its slope decreases to $-\infty$, after which the curve $X_{-1}$ turns to the right. Consequently, in this region our integral curve diverges from the isocline $X_{-1}$ and thus always remains in the domain where $D_->0$.

It remains to verify that our integral curve passes through the domain where $D_+>0$. 
It intersects the left branch of the hyperbola $X_1$ ($D_+=0$) at the stationary point $(u_{01},0)$ when entering the upper half of the strip. 
The hyperbola $X_1$ leaves the point $(u_{01},0)$ with slope $1$ and goes to the right, thus diverging from our integral curve, which continues to move in the domain where $D_+>0$. 
Thus, the only thing we need to check is that our integral curve has no other intersection points with the isocline $X_1$ inside the strip to the left of $(u_{01},0)$. 
Repeating the argument about the impossibility of intersection of the curve $\ell_\eps$ with the isocline $X_{-1}$, we consider two cases: $\eps\le\eps_{\ast}$ and $\eps>\eps_{\ast}$. In the first case, our integral curve is separated from the hyperbola $X_1$ by the isocline $X_{-1}$. 
In the second case, another stationary point appears inside the strip,
\eq{150201}{
\Big(u_0-\frac{a_3^++a_3^-}{a_3^+-a_3^-}\,\eps_\ast,\;\eps_\ast\Big),
}
symmetric to $x_{\ast}$ with respect to the center of symmetry $(u_0,0)$. 
Now, for $x_2\le\eps_{\ast}$ we are still separated from the hyperbola $X_1$ by the isocline $X_{-1}$, while for $x_2\ge\eps_{\ast}$ 
this role is taken over by the isocline $X_0$, which, like the isocline $X_1$, connects the stationary points~\eqref{150201} 
and $(u_{+1},\eps)$. Recall that $u_{+1}$ is the smaller root of the equation $D_+(u,\eps)=0$, i.\,e.,
\eq{150202}{
u_{+1}=-\eps+u_0-\frac{(a_2^+-a_2^-)+
\sqrt{d^2+36(a_3^+-a_3^-)a_3^+\eps^2}}{3(a_3^+-a_3^-)}\,.
}

Figure~\ref{150203} shows the behavior of the isoclines under the assumption $\eps_\ast<\eps<\eps_{\max}$. The behavior of the integral curve $\ell_\eps$ of interest is also depicted there for $\eps>\eps_0$.

\begin{figure}[h]
    \centering
    \includegraphics[scale = 0.3]{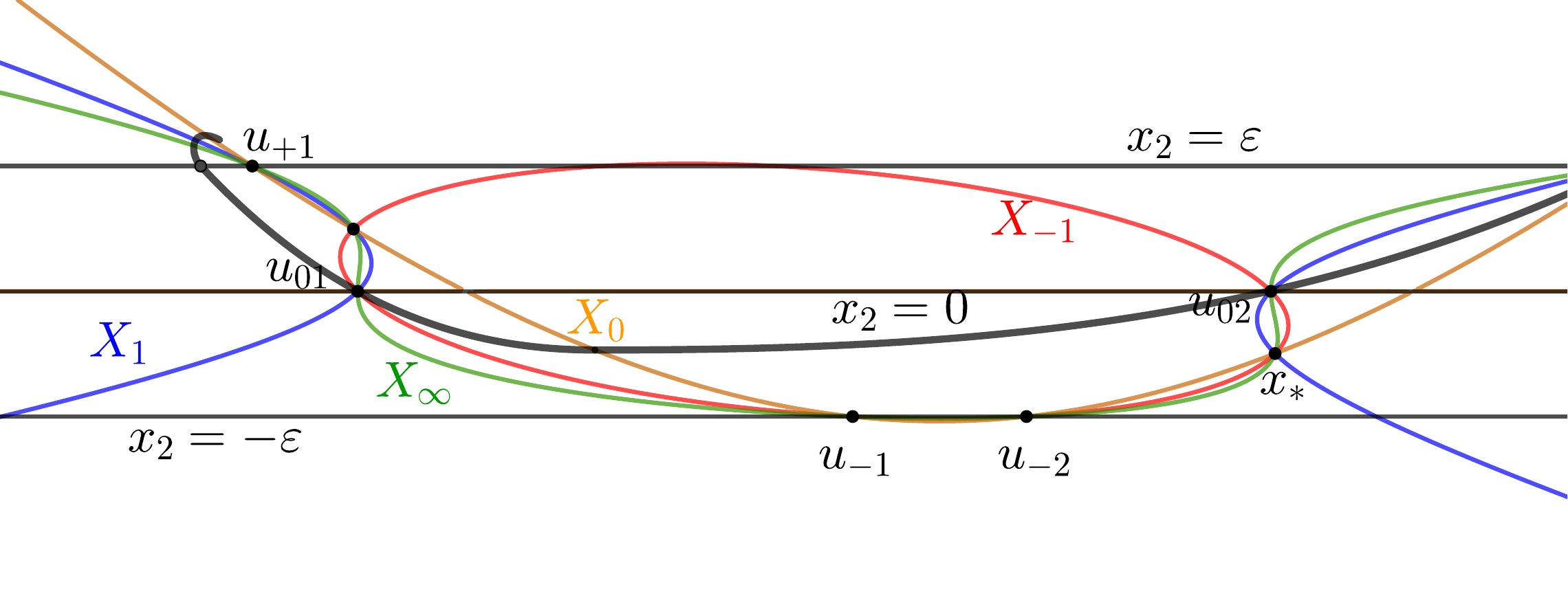}
    \caption{The isoclines $X_0$, $X_{\pm1}$, $X_\infty$ 
    for $\eps_\ast<\eps<\eps_{\max}$.}
    \label{150203}
\end{figure}

The behavior of the isoclines for $\eps>\eps_{\max}$ is almost the same. The only difference is that now the ellipse $X_{-1}$ lies entirely inside the strip, so the stationary points on the lower boundary disappear. Figure~\ref{150204} shows the behavior of the isoclines and the curve $\ell_\eps$ in this case.
\begin{figure}[h]
    \centering
    \includegraphics[scale = 0.3]{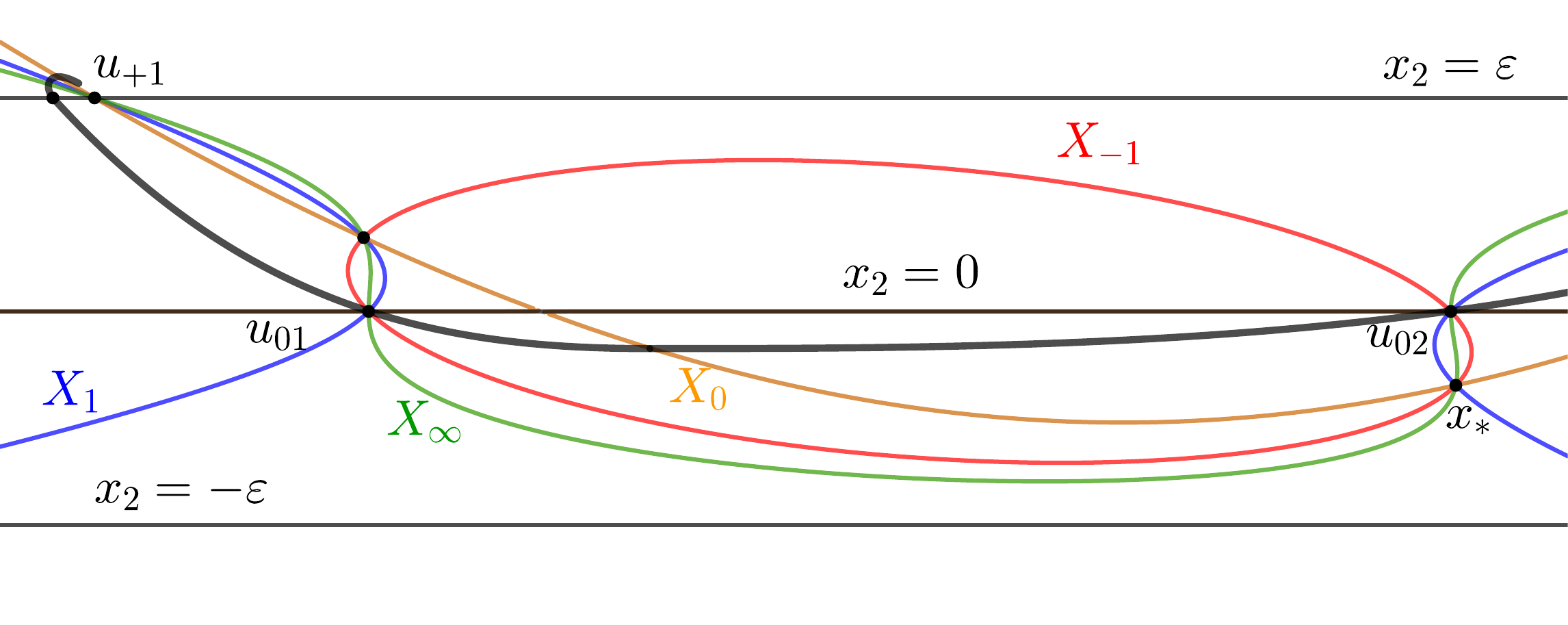}
    \caption{The isoclines $X_0$, $X_{\pm1}$, $X_\infty$ 
    for $\eps>\eps_{\max}$.}
    \label{150204}
\end{figure}

\section{Degenerate stationary points. Medium pockets}

In this section we investigate the situation where the root of the difference of derivatives $f'_+-f'_-$ has multiplicity two. In this case we obtain either the already known foliation with a square of bilinearity\footnote{Recall that we call a function bilinear (although it would be more accurate to call it diagonally linear) if it is linear along the directions $x_1\pm x_2=\const.$} from whose corners oppositely directed horizontal herringbones extend, or a new foliation, which we shall call a {\it medium pocket}. Medium pockets can be regarded as an intermediate foliation between minor and major pockets. As in those cases, this is a horizontal herringbone whose both ends lie on the same boundary of the strip, except that now the spine of this herringbone will be tangent to the midline of the strip. Thus, if the spine of the herringbone forming a minor pocket has no common points with the axis $x_2=0$, while in the case of major pockets there are two such points, then for medium pockets there is one.

Now we formulate the main statement of this section.

\begin{Th} 
\label{130801}
Let $u_0$ be a root of the difference $f'_+-f'_-$ of multiplicity two\textup, i.\,e.\textup,
$f'_+(u_0)=f'_-(u_0),$
$f''_+(u_0)=f''_-(u_0),$
and $f'''_+(u_0)\ne f'''_-(u_0)$. Suppose also that $f'''_+(u_0)\ne0$ and $f'''_-(u_0)\ne0$.

Then\textup, for small $\eps$\textup, in a neighborhood of the point $(u_0,0)$ there arises either a square of bilinearity with two oppositely directed horizontal herringbones\textup, if the values
$f'''_\pm(u_0)$ have opposite signs\textup, or a medium pocket\textup, if the signs of $f'''_\pm(u_0)$ are the same.
\end{Th}

Let us formulate a series of statements detailing the above theorem for different signs of the third derivatives $f'''_\pm(u_0)$. In all of these propositions, the conditions of the theorem will be assumed to hold, i.\,e., $u_0$ will be a multiple root of the difference $f'_+-f'_-$.

\begin{Prop}
\label{130802}
If $f'''_+(u_0)>0$ and $f'''_-(u_0)<0,$ then for small $\eps$ near the point $x_1=u_0$ one can construct the following foliation\textup:
\eq{140801}{
\begin{aligned}
\Om{R}(u_1,u_-^l)&\cup\Om{HB}\big((u_0-\eps,0),(u_-^l-\eps,-\eps)\big)\cup\Om{Rect}(u_0+\eps,0)\cup
\\
&\cup\Om{HB}\big((u_0+\eps,0),(u_+^r+\eps,\eps)\big)\cup\Om{R}(u_+^r,u_2)
\end{aligned}
}

{\begin{itemize}
    \item a square with vertices at the points $(u_0,\pm\eps)$ and $(u_0\pm\eps,0)$\textup, on which the function is bilinear \textup($\Om{Rect}(u_0+\eps,0)$\textup)\textup;
    \item from the right corner of the square there extends a left horizontal herringbone with the tip on the upper boundary at the point $(u_+^r+\eps,\eps),$ where $u_+^r$ is the root of the equation $D_+(u,\eps)=0$ closest to $u_0$ on the right\textup; this herringbone foliates the domain $\Om{HB}\big((u_0+\eps,0),(u_+^r+\eps,\eps)\big)$\textup;
    \item from the left corner of the square there extends a right horizontal herringbone with the tip on the lower boundary at the point $(u_-^l-\eps,-\eps),$ where $u_-^l$ is the root of the equation $D_-(u,\eps)=0$ closest to $u_0$ on the left\textup; this herringbone foliates the domain $\Om{HB}\big((u_0-\eps,0),(u_-^l-\eps,-\eps)\big)$\textup;
    \item to the left of the point $(u_-^l,0)$ there is a right simple foliation\textup, $\Om{R}(u_1,u_-^l)$\textup;
    \item to the right of the point $(u_+^r,0)$ there is a right simple foliation\textup, $\Om{R}(u_+^r,u_2)$.
\end{itemize}}  
\end{Prop}

\begin{figure}[h]
    \centering
    \includegraphics[scale = 0.4]{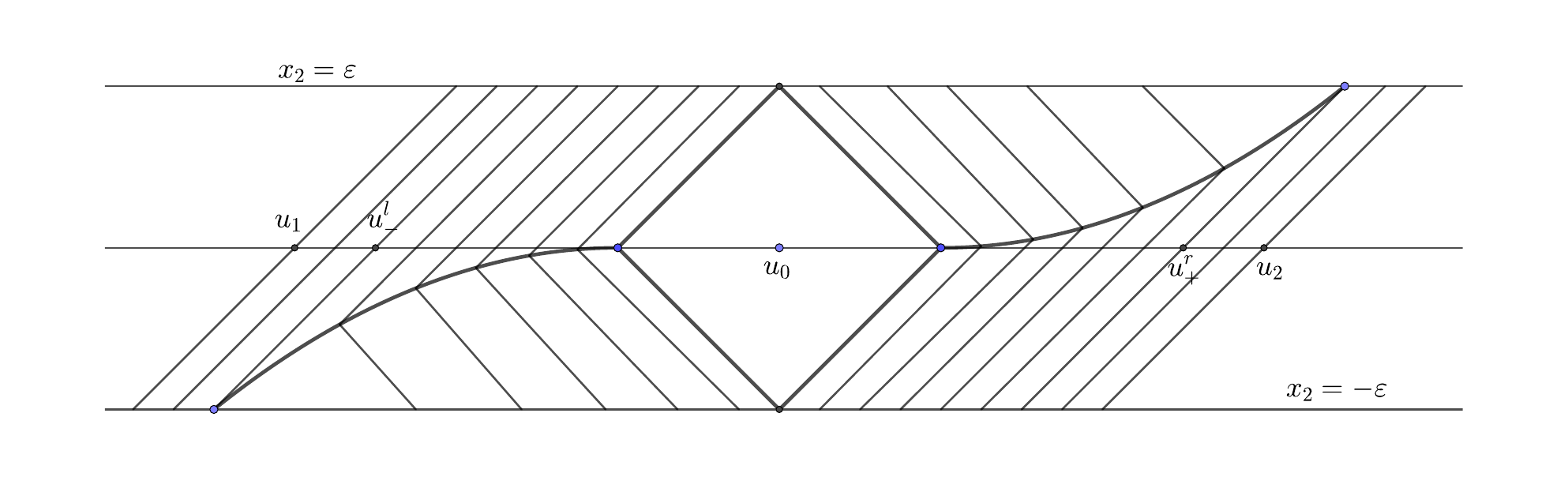}
    \caption{The foliation in Proposition~\ref{130802}.}
    \label{140802}
\end{figure}

Symmetry with respect to either of the axes leads to the following statement.

\begin{Prop}
\label{130803}
If $f'''_+(u_0)<0$ and $f'''_-(u_0)>0,$ then for small $\eps$ near the point $x_1=u_0$ one can construct the following foliation\textup:
\eq{140803}{
\begin{aligned}
\Om{L}(u_1,u_+^l)&\cup\Om{HB}\big((u_0\!-\!\eps,0),(u_+^l\!\!-\!\eps,\eps)\big)\cup\Om{Rect}(u_0\!+\!\eps,0)\cup
\\
&\cup\Om{HB}\big((u_0\!+\!\eps,0),(u_-^r\!\!+\!\eps,-\eps)\big)\cup\Om{L}(u_-^r,u_2)
\end{aligned}
}

{\begin{itemize}
    \item a square with vertices at the points $(u_0,\pm\eps)$ and 
    $(u_0\pm\eps,0)$\textup, on which the function is bilinear 
    \textup($\Om{Rect}(u_0+\eps,0)$\textup)\textup;
    \item from the right corner of the square there extends a left horizontal 
    herringbone with the tip on the lower boundary at the point $(u_-^r\!+\eps,-\eps),$ 
    where $u_-^r$ is the root of the equation $D_-(u,-\eps)=0$ closest to $u_0$ on the right\textup; this herringbone foliates the domain  
    $\Om{HB}\big((u_0+\eps,0),(u_-^r\!+\eps,-\eps)\big)$\textup;
    \item from the left corner of the square there extends a right horizontal 
    herringbone with the tip on the upper boundary at the point $(u_+^l+\eps,\eps),$ 
    where $u_+^l$ is the root of the equation $D_+(u,-\eps)=0$ closest to $u_0$ on the left\textup; this herringbone foliates the domain 
    $\Om{HB}\big((u_0-\eps,0),(u_+^l-\eps,\eps)\big)$\textup;
    \item to the left of the point $(u_+^l,0)$ there is a left simple 
    foliation\textup, $\Om{L}(u_1,u_+^l)$\textup;
    \item to the right of the point $(u_-^r,0)$ there is a left simple 
    foliation\textup, $\Om{L}(u_-^r,u_2)$.
\end{itemize}}  
\end{Prop}

\begin{figure}[h]
    \centering
    \includegraphics[scale = 0.4]{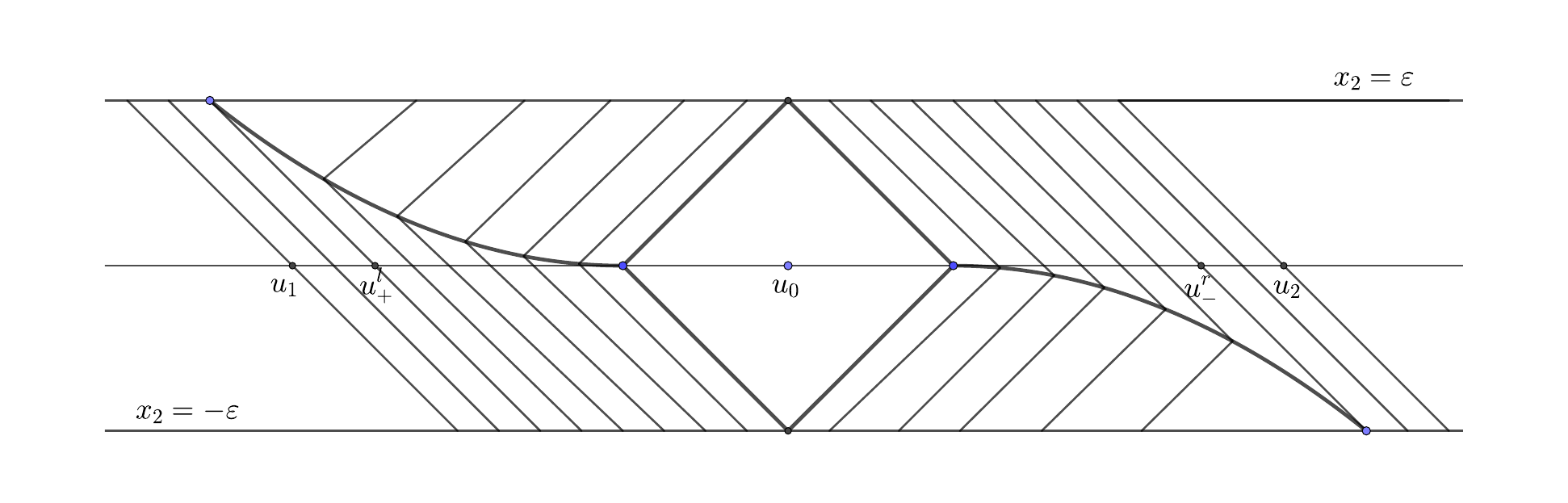}
    \caption{The foliation in Proposition~\ref{130803}.}
    \label{140802a}
\end{figure}

Now we formulate the propositions describing the foliation in the cases where the third derivatives of the boundary functions have the same sign at $u_0$.

\begin{Prop}
\label{130804}
If $f'''_+(u_0)>f'''_-(u_0)>0,$ then for small $\eps$, from some point $(v_+^l+\eps,\eps)$ on the upper boundary there extends a left herringbone that passes through the upper half of the strip\textup, is tangent to the axis $x_1$ at the point $(u_0+\eps,0)$\textup, and ends on the upper boundary at the point $(u_+^r+\eps,\eps),$ where $u_+^r$ is the root of the equation $D_+(u,\eps)=0$ closest to $u_0$ on the right. Thus\textup, the following foliation can be constructed\textup:
\eq{170802a}{
\Om{R}(u_1,v_+^l)\cup\Om{HB}\big((v_+^l\!+\!\eps,\eps),(u_+^r\!\!+\!\eps,\eps)\big)
\cup\Om{R}(u_+^r,u_2).
}
\end{Prop}

\begin{figure}[h]
    \centering
    \includegraphics[scale = 0.25]{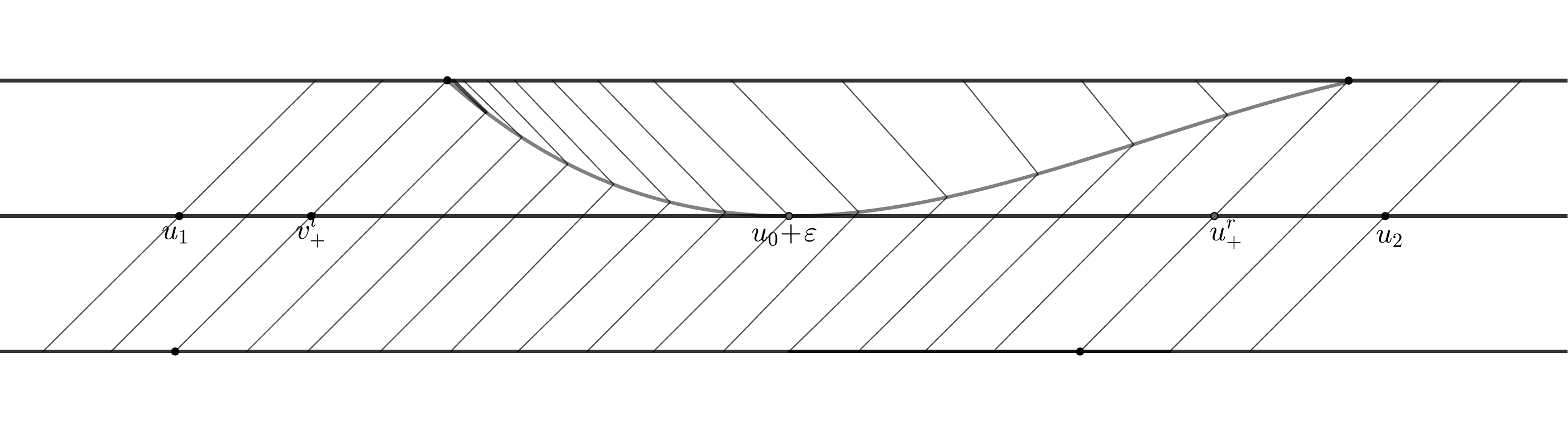}
    \caption{The foliation in Proposition~\ref{130804}.}
    \label{170801}
\end{figure}

The symmetric statement with respect to the axis $x_1=0$ has the following form.

\begin{Prop}
\label{130805}
If $f'''_+(u_0)<f'''_-(u_0)<0,$ then for small $\eps$\textup, from some point $(v_+^r-\eps,\eps)$ on the upper boundary there extends a right herringbone that passes through the upper half of the strip\textup, is tangent to the axis $x_2=0$ at the point $(u_0-\eps,0)$\textup, and ends on the upper boundary at the point $(u_+^l-\eps,\eps),$ where $u_+^l$ is the root of the equation $D_+(u,-\eps)=0$ closest to $u_0$ on the left. Thus\textup, the following foliation can be constructed\textup:
\eq{170802}{
\Om{L}(u_1,u_+^l)\cup\Om{HB}\big((v_+^r\!\!-\!\eps,\eps),(u_+^l\!-\!\eps,\eps)\big)
\cup\Om{L}(v_+^r,u_2).
} 
\end{Prop}

\begin{figure}[h]
    \centering
    \includegraphics[scale = 0.25]{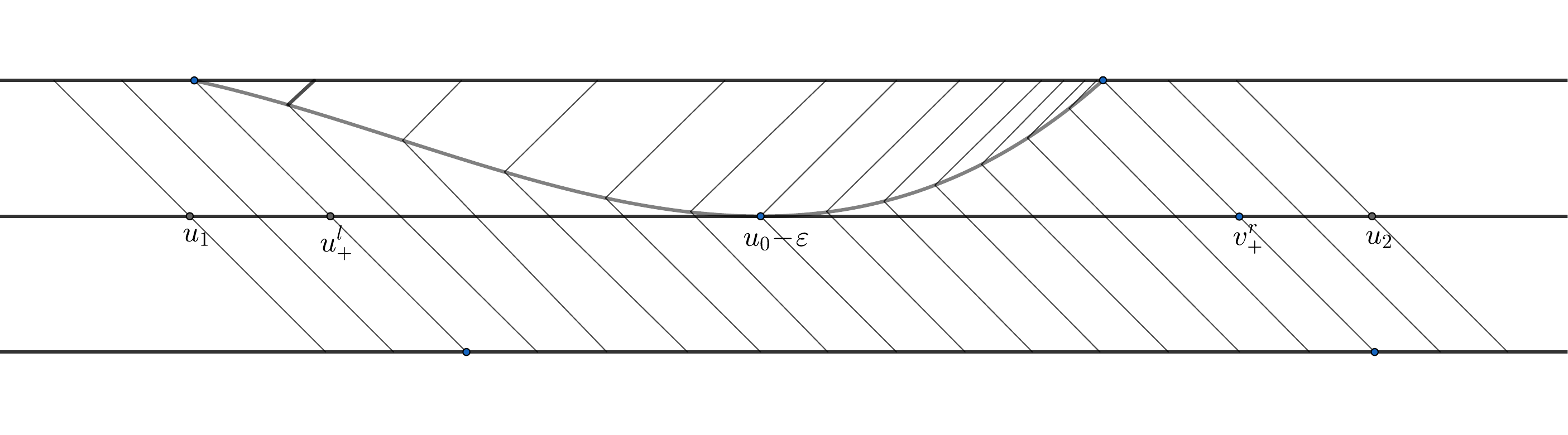}
    \caption{The foliation in Proposition~\ref{130805}.}
    \label{240801}
\end{figure}

A pair of statements obtained by reflection with respect to the axis $x_2=0$ is contained in the following two propositions.

\begin{Prop}
\label{130806}
If $f'''_-(u_0)>f'''_+(u_0)>0,$ then for small $\eps$\textup, from some point $(v_-^l\!+\eps,-\eps)$ on the lower boundary there extends a left herringbone that passes through the lower half of the strip\textup, is tangent to the axis $x_2=0$ at the point $(u_0+\eps,0)$\textup, and ends on the lower boundary at the point $(u_-^r\!+\eps,-\eps),$ where $u_-^r$ is the root of the equation $D_-(u,-\eps)=0$ closest to $u_0$ on the right. Thus\textup, the following foliation can be constructed\textup:
\eq{131001}{
\Om{L}(u_1,v_-^l)\cup\Om{HB}\big((v_-^l\!\!+\!\eps,-\eps),(u_-^r\!\!+\!\eps,-\eps)\big)
\cup\Om{L}(u_-^r,u_2).
} 
\end{Prop}
\begin{figure}[h]
    \centering
    \includegraphics[scale = 0.10]{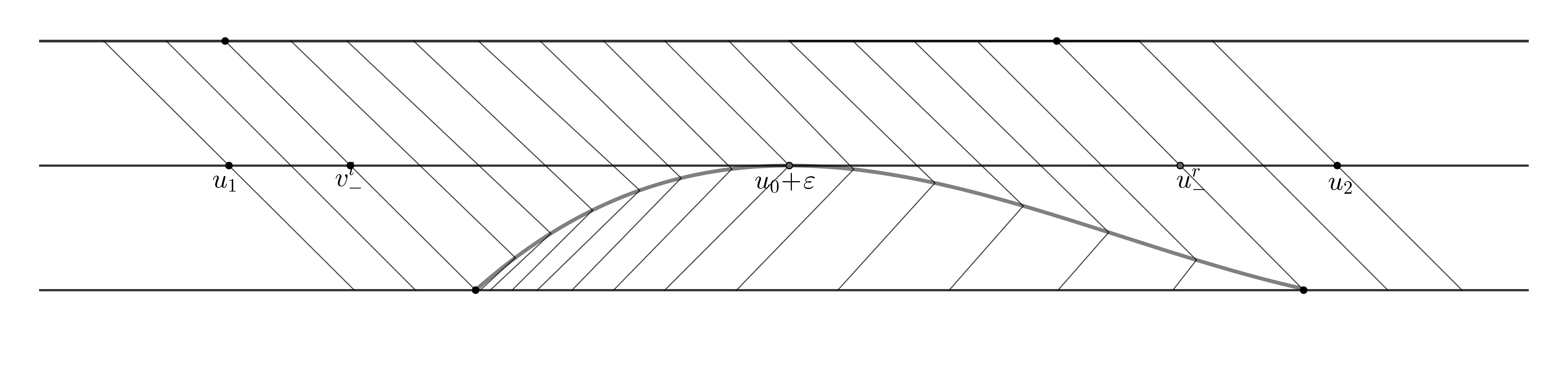}
    \caption{The foliation in Proposition~\ref{130806}.}
    \label{131002}
\end{figure}

\begin{Prop}
\label{130807}
If $f'''_-(u_0)<f'''_+(u_0)<0,$ then for small $\eps$\textup, from some point $(v_-^r\!-\eps,-\eps)$ on the lower boundary there extends a right herringbone that passes through the lower half of the strip\textup, is tangent to the axis $x_2=0$ at the point $(u_0-\eps,0)$\textup, and ends on the lower boundary at the point $(u_-^l\!-\eps,-\eps),$ where $u_-^l$ is the root of the equation $D_-(u,\eps)=0$ closest to $u_0$ on the left. Thus\textup, the following foliation can be constructed\textup:
\eq{131003}{
\Om{R}(u_1,v_-^l)\cup\Om{HB}\big((v_-^l\!\!+\!\eps,-\eps),(u_-^r\!\!+\!\eps,-\eps)\big)
\cup\Om{R}(u_-^r,u_2).
} 
\end{Prop}

\begin{figure}[h]
    \centering
    \includegraphics[scale = 0.10]{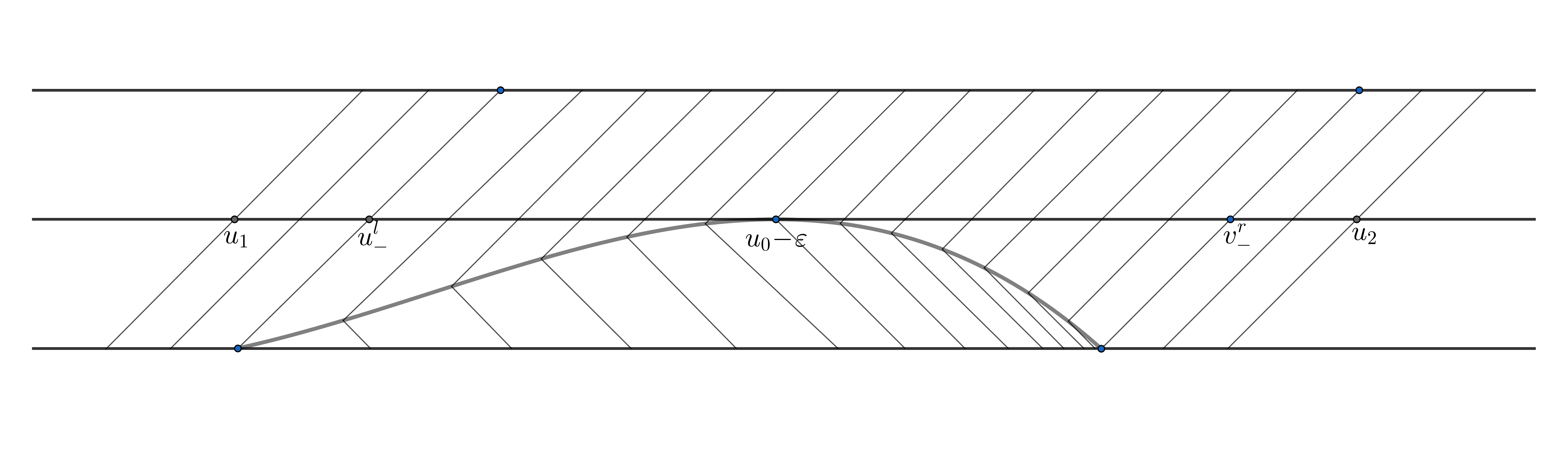}
    \caption{The foliation in Proposition~\ref{130807}.}
    \label{131004}
\end{figure}

Before proving the stated propositions, we prove the following lemma.

\begin{Le}
\label{220201}   
Let $f'_+(u_0)=f'_-(u_0),$ $f''_+(u_0)=f''_-(u_0),$ $f'''_+(u_0)>0$
and $f'''_+(u_0)>f'''_-(u_0)$. Then, for sufficiently small $\eps$\textup, there exists a root $u_+$ of the equation $D_+(u,\eps)=0$ such that one can construct a left herringbone extending from the point $(u_0+\eps,\,0)$ to the point $(u_+\!+\eps,\,\eps).$
\end{Le}

\begin{proof}
Consider the behavior of the function $D_+$ in a neighborhood of the point $(u_0,0)$. To this end, write $x_1=u_0+(\alpha+1)x_2$ and expand the resulting function in powers of $x_2$, treating $\alpha$ as a fixed parameter:
\eq{230201}{
\begin{aligned}
2x_2 D_+(x_1,x_2)&=f'_+(u_0+(\alpha+2)x_2)-f'_-(u_0+\alpha x_2)
-2x_2 f''_+(u_0+(\alpha+2)x_2)
\\
&=\Big[\half\alpha^2\big(f'''_+(u_0)-f'''_-(u_0)\big)-2f'''_+(u_0)\Big]x_2^2
+O(x_2^3)\,.
\end{aligned}
}
We see that the isocline $X_1$ ($D_+=0$) has two branches in a neighborhood of $(u_0,0)$. They intersect at this point and have slopes
\eq{230202}{
\frac1{1\pm\alpha_0},\qquad \text{where}\quad 
\alpha_0=2\;\sqrt\frac{f'''_+(u_0)}{f'''_+(u_0)-f'''_-(u_0)}\,,
}
Thus, on the boundary $x_2=\eps$ the function $D_+$ has two roots in a neighborhood of $u_0$. Choose the larger one:
\eq{230203}{
u_+=u_0+(\alpha_0+1)\eps+O(\eps^2)\,,
}
and verify the conditions of Theorem~\ref{170301} for the point $u=u_+$. Condition~\eqref{170302} holds by the definition of $u_+$. Condition~\eqref{170304} holds for small $\eps$ since $f'''_+(u_0)>0$. Finally, condition~\eqref{170303} is the condition $\frac{\partial}{\partial x_1}D_+(u_+,\eps)>0$, which holds because the function $D_+(x_1,\eps)$ is negative to the left of the simple root $u_+$ and positive to the right of it.

It remains to verify that the resulting herringbone can be continued up to the point $(u_0+\eps,0)$. In other words, we need to trace that the integral curve $\ell_\eps$ defining the spine of the herringbone passes through the domain where $D_+>0$ and $D_->0$ and terminates at the point $(u_0,0)$. As before, we consider the subdomain $\omega$ in the upper half-strip bounded by the line $x_1=u_+$ and that part of the isocline $D_+=0$ which connects the points $(u_0,0)$ and $(u_+,\eps)$.

We shall essentially repeat the arguments used in the proof of Proposition~6.1 (see~\cite{First}). First, note that for small $\eps$ the slope of the chosen branch of the isocline is close to $1/(1+\alpha_0)$, while the integral curves have slope equal to $1$ at points of the isocline. Therefore, the curve $\ell_\eps$ in question, starting downward from the point $(u_+,\eps)$ to the right of the isocline, cannot intersect it at any point of the upper half-strip. Nor can the curve $\ell_\eps$ leave the domain $\omega$ through the line $x_2=0$ except through the stationary point $(u_0,0)$. Second, in $\omega$ we have $D_+>0$ and $D_->0$. The first inequality holds because, by construction, $\omega$ lies to the right of the chosen segment of the isocline $D_+=0$, on which the function $D_+$ changes sign from minus to plus when we intersect this isocline from left to right. We now verify the second inequality by examining the values of $D_-$ on the rays $x_1=u_0+(\alpha+1)x_2$:
\eq{260201}{
\begin{aligned}
2x_2 D_-(x_1,x_2)&=f'_+(u_0+(\alpha+2)x_2)-f'_-(u_0+\alpha x_2)
-2x_2 f''_-(u_0+\alpha x_2)
\\
&=\Big[(\half\alpha^2+2\alpha)\big(f'''_+(u_0)-f'''_-(u_0)\big)
+2f'''_+(u_0)\Big]x_2^2+O(x_2^3)\,.
\end{aligned}
}
Since $\alpha>0$ for points in $\omega$, this expression is strictly positive in $\omega$. Thus we have verified that the curve $\ell_\eps$ starting from $(u_+,\eps)$ moves leftward and downward inside the domain where $D_+>0$ and $D_->0$, eventually reaching the point $(u_0,0)$. Consequently, one can construct a left herringbone extending from the point $(u_0+\eps,0)$ to the point $(u_++\eps,\eps)$, as required.
\end{proof}

In order to be able to refer to the symmetric variants of the lemma just proved, we formulate them separately.

\begin{Le}
\label{260202}   
Let $f'_+(u_0)=f'_-(u_0),$ $f''_+(u_0)=f''_-(u_0),$ $f'''_-(u_0)>0,$
and $f'''_+(u_0)<f'''_-(u_0)$. Then\textup, for sufficiently small $\eps$\textup, there exists a root $u_-$ of the equation $D_-(u,-\eps)=0$ such that one can construct a left herringbone extending from the point $(u_0+\eps,\,0)$ to the point $(u_-\!+\eps,-\eps).$
\end{Le}

\begin{Le}
\label{260203}   
Let $f'_+(u_0)=f'_-(u_0),$ $f''_+(u_0)=f''_-(u_0),$ $f'''_+(u_0)<0,$
and $f'''_+(u_0)<f'''_-(u_0)$. Then\textup, for sufficiently small $\eps$\textup, there exists a root $u_-$ of the equation $D_+(u,-\eps)=0$ such that one can construct a right herringbone extending from the point $(u_0-\eps,\,0)$ to the point $(u_-\!-\eps,\,\eps).$
\end{Le}

\begin{Le}
\label{260204}   
Let $f'_+(u_0)=f'_-(u_0),$ $f''_+(u_0)=f''_-(u_0),$ $f'''_-(u_0)<0,$
and $f'''_+(u_0)>f'''_-(u_0)$. Then\textup, for sufficiently small $\eps$\textup, there exists a root $u_+$ of the equation $D_-(u,\eps)=0$ such that one can construct a right herringbone extending from the point $(u_0-\eps,\,0)$ to the point $(u_+\!-\eps,-\eps).$
\end{Le}

We are now ready to prove Propositions~\ref{130802}--\ref{130807}.
Since Propositions~\ref{130802} and~\ref{130803} are symmetric,
we give the proof only of the first one.

\begin{proof}[Proof of Proposition~\textup{\ref{130802}}]
Since $f'''_+(u_0)>0$ and $f'''_-(u_0)<0$, the conditions of Lemmas~\ref{220201} and~\ref{260204} are satisfied simultaneously. Thus we can immediately construct two herringbones: a left one from the point $(u_0+\eps,\,0)$ and a right one from the point $(u_0-\eps,\,0)$. The proof is completed by applying Proposition~11.1 (see~\cite{Second}), whose conditions are satisfied with $C=(u_0,0)$. Namely, the two herringbones constructed can be joined by a rectangle, which in this case is a square.
\end{proof}

We now prove one of the symmetric Propositions~\ref{130804}--\ref{130807}.

\begin{proof}[Proof of Proposition~\textup{\ref{130804}}]
As in the previous proof, using Lemma~\ref{220201}, we construct a left herringbone from the point $(u_0+\eps,\,0)$ to the point $(u_+\!+\eps,\,\eps)$. Now we need to continue it to the left up to the second intersection with the upper boundary. In other words, we need to continue the curve $\ell_\eps$ to the left from the point $(u_0,0)$ so that it runs in the upper half-strip up to the upper boundary, remaining in the domain where $D_+>0$ and $D_->0$. To this end, we need to study in more detail the behavior of the isoclines of the integral curves in a neighborhood of $(u_0,0)$.

Substituting the expansion
\eq{270201}{
x_2=k(x_1-u_0)+O(x_2^2)
}
into the equations of our field \textup{\bf(}e.\,g., into formula~(5.3)\textup{\bf)}, we obtain the following equation for the slope coefficient $k$:
\eq{280201}{
k=\frac{2k}{1-k^2}\cdot
\frac{(1+k)f'''_+(u_0)-(1-k)f'''_-(u_0)}{f'''_+(u_0)-f'''_-(u_0)}\,.
}
All integral curves except two correspond to the value $k=0$, i.\,e., they are tangent to the axis $x_1$. The two distinguished integral curves, which are separatrices, have slopes
\eq{280202}{
k_\pm=-\frac{\sqrt{f'''_+(u_0)}\pm\sqrt{f'''_-(u_0)}}
{\sqrt{f'''_+(u_0)}\mp\sqrt{f'''_-(u_0)}}\,.
}
The sector between the separatrices is foliated by integral curves that do not pass through $(u_0,0)$. The behavior of the extremal curves in a neighborhood of this point is shown in Figure~\ref{280203}.
\begin{figure}[h]
    \centering
    \includegraphics[scale = 0.3]{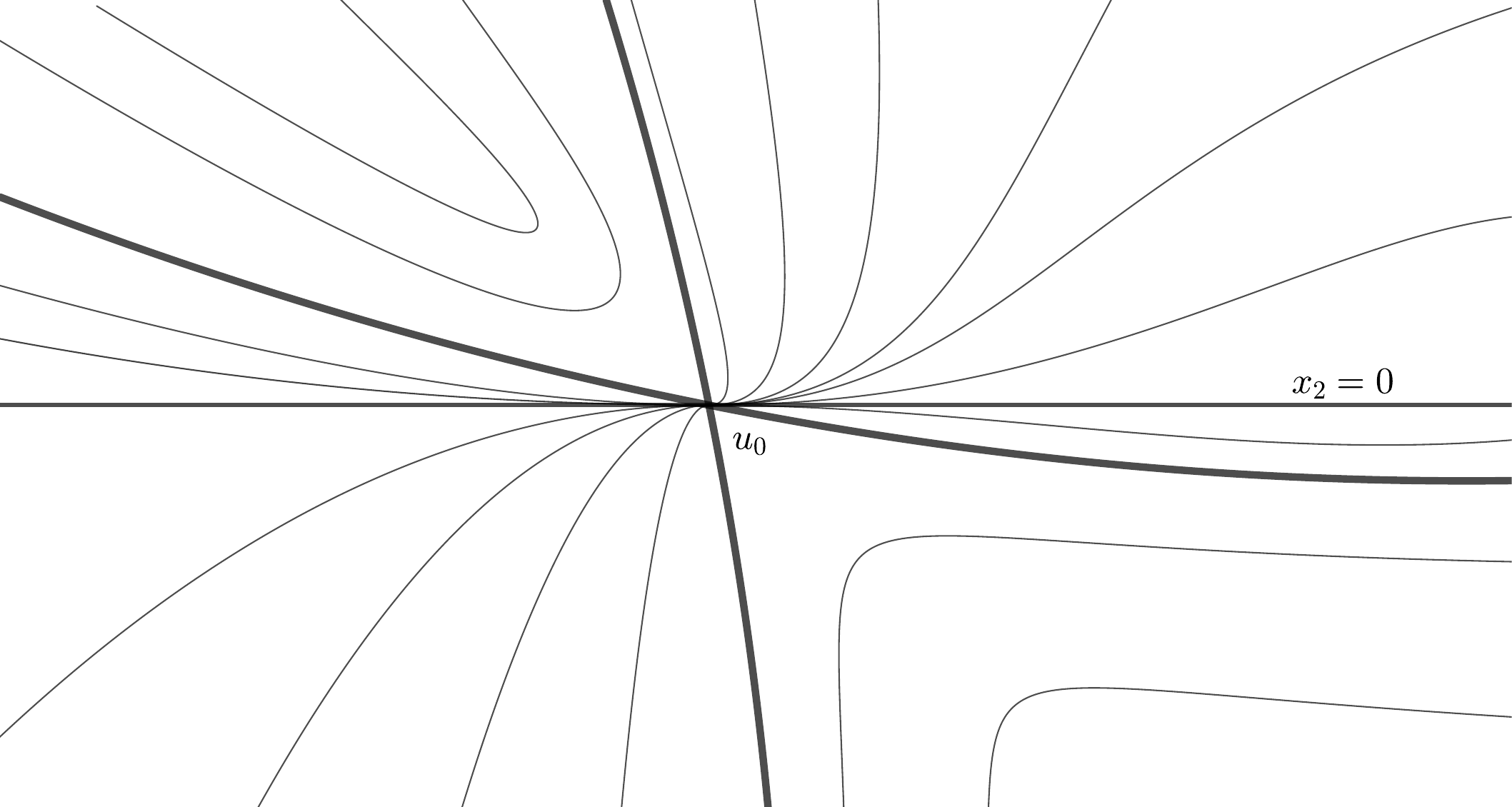}
    \caption{Integral curves of the field in a neighborhood of the stationary 
    point $(u_0,0)$.}
    \label{280203}
\end{figure}

Thus, our integral curve $\ell_\eps$ is tangent to the axis $x_1$ from the right at the point $(u_0,0)$, and we can continue it by any integral curve tangent to the axis $x_1$ at $(u_0,0)$ from the left in order to obtain a $C^1$-smooth curve. This curve suplies us with the spine of a herringbone, provided that it passes through the strip in the domain where $D_+>0$ and $D_->0$. We need not worry about the inequality $D_->0$ at all, since under the conditions stated in Proposition~\ref{130804} it holds in some neighborhood of $(u_0,0)$ in the upper half-plane. To verify this, it suffices to rewrite equality~\eqref{260201} in the form
\eq{010301}{
2x_2 D_-(x_1,x_2)=\Big[\half(\alpha+2)^2\big(f'''_+(u_0)-f'''_-(u_0)\big)
+2f'''_-(u_0)\Big]x_2^2+O(x_2^3)\,.
}
The fact that for sufficiently small $\eps$ we can always find an integral curve tangent to the axis $x_1$ at $(u_0,0)$ from the left on which the condition $D_+>0$ holds follows simply by choosing the initial point $v_+^l$ of this curve on the boundary $x_2=\eps$ to the left of the left root of the equation $D_+(u,\eps)=0$.
\end{proof}

Note that we have proved the proposition by showing the possibility of constructing a medium pocket in the described situation, but we have not fully determined the integral curve defining the spine of the herringbone. The final choice of this curve requires additional investigation, which we shall not carry out here. We only put forward a conjecture: the natural candidate for the continuation of the curve $\ell_\eps$ from the point $(u_0,0)$ to the left is the curve that joins at $(u_0,0)$ in a $C^2$-smooth manner.

On the other hand, since the simple foliation gives the minimal possible value of the diagonally concave function we are looking for, it is desirable to take the point $v_+^l$ as far to the right as possible. The maximum value is attained on the separatrix, but we cannot take it for constructing the spine of the herringbone because it has negative slope at $(u_0,0)$, i.\,e., it does not join the curve $\ell_\eps$ in a $C^1$-smooth manner, which is necessary for constructing the foliation.

\section{Examples. Third-degree polynomials: degenerate cases}
\label{141000}

In this section we complete the study of foliations arising in the case where the boundary values are third-degree polynomials. Recall that, in order to avoid considering cases obtained by symmetry from those already treated, we assume that
\eq{141001}{
a_3^+\ge |a_3^-|.
}
The nondegenerate case is defined as the one where the discriminant of the polynomial $f'_+-f'_-$ is nonzero and $a_3^-\ne0$. We add that we always have $a_3^+>0$, since otherwise condition~\eqref{141001} would imply $a_3^+=a_3^-=0$, i.\,e., the boundary polynomials would have degree less than three, which is not the subject of our consideration.

The nondegenerate cases were fully treated in Sections~7, 9, 10, 12, and~\ref{141003}. Moreover, the degenerate cases have already been partially treated as well. In a certain sense, the case where the function $f'_+-f'_-$ is linear, i.\,e., $a_3^+=a_3^-$, can be regarded as degenerate. This case was studied in Section~7 (see~\cite{First}). Furthermore, in Section~12 (see~\cite{Second}) the case where the discriminant of the quadratic polynomial $f'_+-f'_-$ vanishes and $a_3^-<0$ was treated. When the discriminant of the quadratic polynomial $f'_+-f'_-$ is negative, the case $a_3^-=0$ was not distinguished in any way, and Section~9 gives a description of the foliation common to all nonnegative $a_3^-$. Thus, we are left with three degenerate cases:
\begin{itemize}
\item the discriminant of the quadratic polynomial $f'_+-f'_-$ is positive and $a_3^-=0$;
\item the discriminant of the quadratic polynomial $f'_+-f'_-$ vanishes and $a_3^->0$;
\item the discriminant of the quadratic polynomial $f'_+-f'_-$ vanishes and $a_3^-=0$.
\end{itemize}

Let us begin with the first case, where the polynomial $f'_+-f'_-$ has two distinct real roots. Then, for a nonzero coefficient $a_3^-$, two fissures appear in a neighborhood of these roots for small $\eps$ (see Figs.~14 and~15 in~\cite{First}). On the right there was always an SW-fissure, while the direction of the fissure on the left depended on the sign of $a_3^-$: in the case $a_3^->0$ we used Proposition~6.3 and constructed an NW-fissure, whereas for $a_3^-<0$ we used Proposition~6.4 and constructed an NE-fissure. Now, for $a_3^-=0$, we cannot construct any fissure to connect the domains of right and left simple foliations. However, since the boundary function on the lower boundary is a quadratic polynomial, we can connect these domains by a triangle of bilinearity $\Om{T}(u_{01})$,
\eq{151001}{
\Om{T}(u)\df\{x\in\Oe\colon u+x_2-\eps\le x_1\le u-x_2+\eps\}.
}

Thus, let us verify that for small $\eps$ we have the foliation
\eq{161001}{
\begin{aligned}
\Oe=&\;
\Om{R}(-\infty,u_{01}-\eps)\cup\Om{T}(u_{01})\cup\Om{L}
(u_{01}+\eps,v_{-2})
\\
\quad&
\cup\Om{HB}\big((v_{-2}\!+\eps,-\eps),(u_{+2}\!+\eps,\eps)\big)
\cup\Om{R}(u_{+2},\infty),
\end{aligned}
}
which is shown in Figure~\ref{241002}.
\begin{figure}[h]
    \centering
    \includegraphics[scale = 0.35]{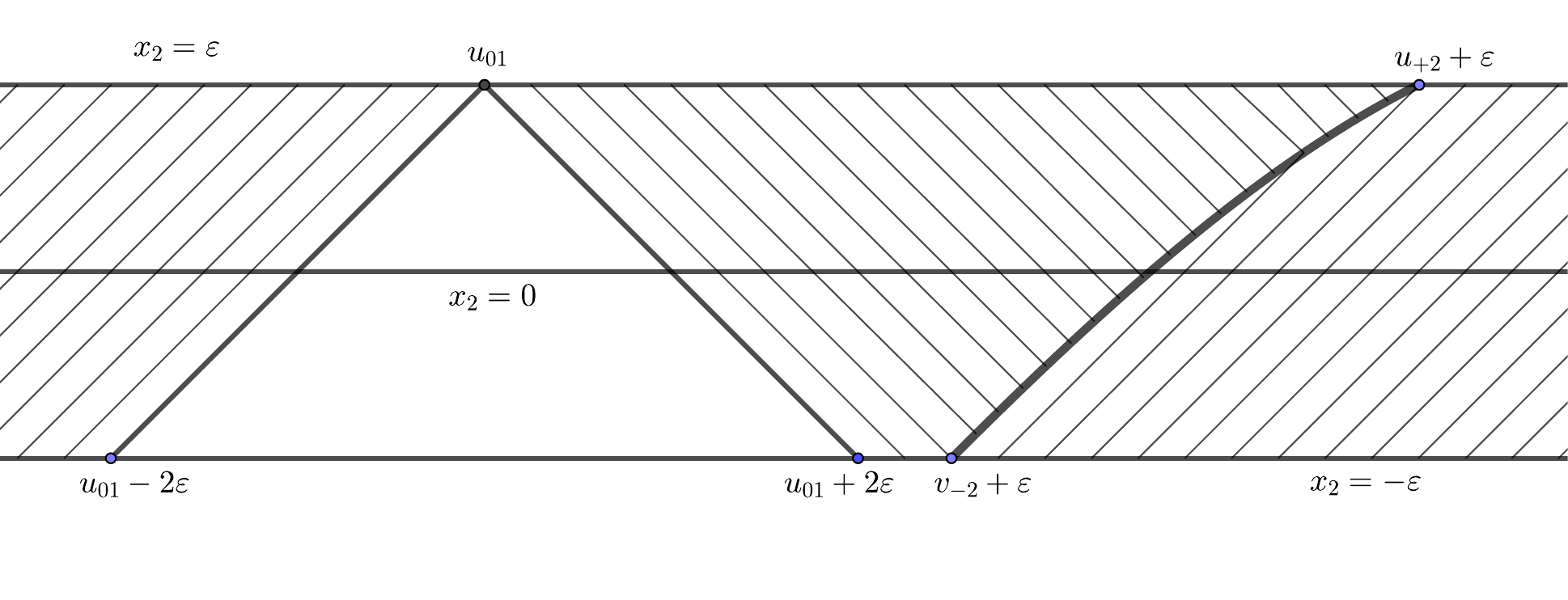}
    \caption{The foliation~\eqref{161001}.}
    \label{241002}
\end{figure}

Note that in our special notation for fissures introduced in Section~6 (see~\cite{First}),
$$
\Om{HB}\big((v_{-2}\!+\eps,-\eps),(u_{+2}\!+\eps,\eps)\big)=
\Om{SW}(v_{-2},u_{+2}),
$$
but we have used the general notation for horizontal herringbones because, as we shall see, in the course of evolution (i.\,e., as $\eps$ grows) the base of this herringbone will detach from the lower boundary and it will cease to be a fissure.

To prove that the described foliation corresponds to the given boundary values, we need to verify that the inequalities $D_+\ge0$ and $D_-\ge0$ hold on the upper boundary of the strip for $x_1\le u_{01}-\eps$ and $x_1\ge u_{+2}$, as well as on the lower boundary for $u_{01}+\eps\le x_1\le v_{-2}$. One can visually verify these inequalities by looking at Figure~\ref{161002}, which shows the components of the isoclines $X_1$ and $X_{-1}$ that are the graphs of the curves $D_+=0$ and $D_-=0$. (Recall that for the field giving a right herringbone, the upper boundary of the strip always belongs to the isocline $X_{-1}$, while the lower boundary belongs to $X_1$.)
\begin{figure}[h]
    \centering
    \includegraphics[scale = 0.4]{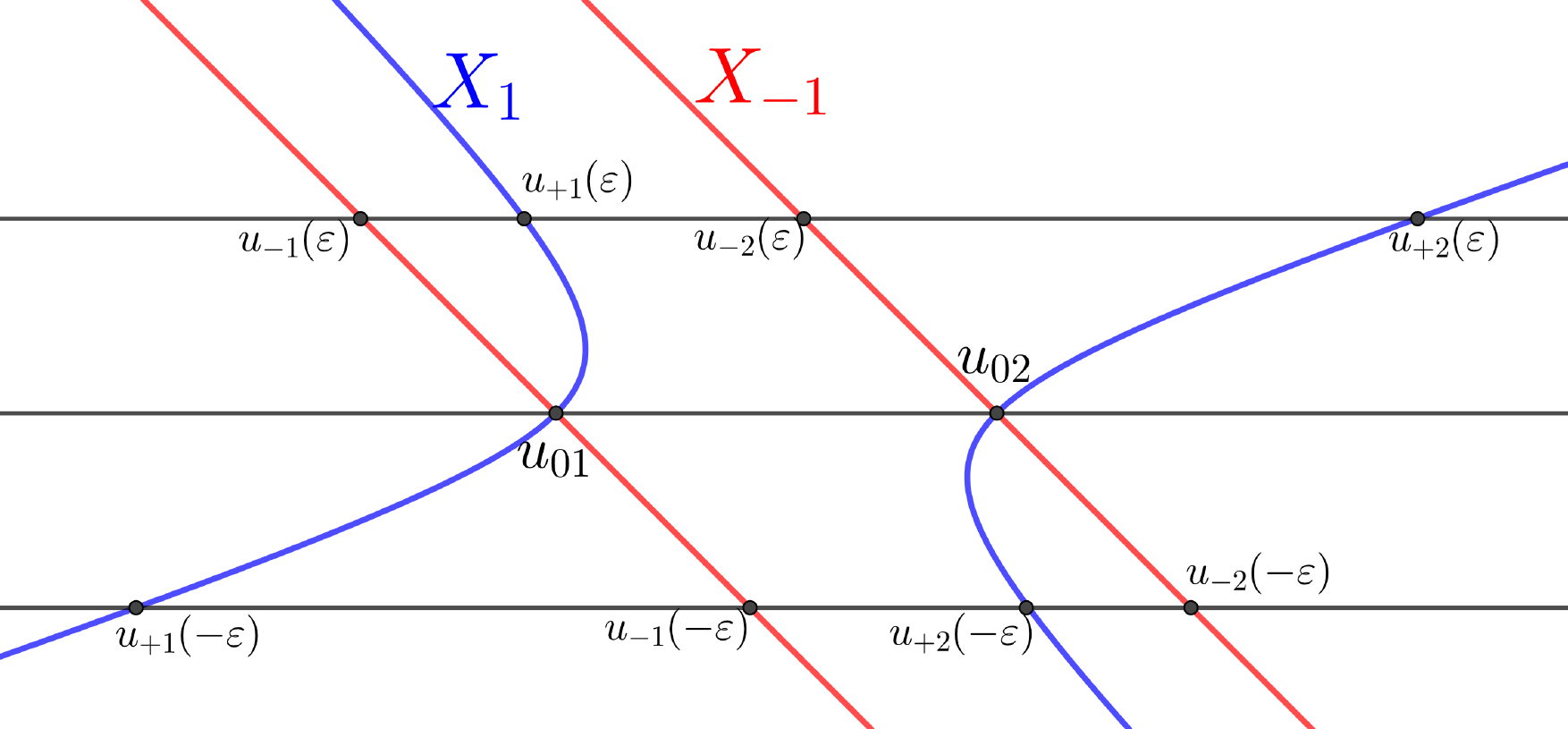}
    \caption{The graphs of the curves $D_+=0$ and $D_-=0$.}
    \label{161002}
\end{figure}

Since $D_+$ and $D_-$ are quadratic polynomials in the variable $x_1$, to determine their sign it suffices to compute the roots of these polynomials. We shall use formulas~(9.3)--(9.4) from \cite{Second}, where the expressions for $D_+$ and $D_-$ for third-degree polynomials were computed. Since $a_3^-=0$, we have
\eq{290301}{
2x_2D_-(x)=3a_3^+(x_1+x_2-u_0)^2-\frac{d^2}{3a_3^+}=
3a_3^+(x_1+x_2-u_{01})(x_1+x_2-u_{02})\,,
}
where
\eq{290302}{
\begin{gathered}
    u_0=-\frac{a_2^+-a_2^-}{3a_3^+}\,,\qquad d=\sqrt{(a_2^+-a_2^-)^2-3a_3^+(a_1^+-a_1^-)}\,,
    \\
    u_{01}=u_0-\frac{d}{3a_3^+}\,,\qquad u_{02}=u_0+\frac{d}{3a_3^+}\,.
\end{gathered}
}
Therefore, the isocline $D_-=0$ splits into two straight lines: $x_1+x_2=u_{01}$ and $x_1+x_2=u_{02}$. We need the roots on the upper boundary:
\eq{211001}{
u_{-1}(\eps)=u_0-\frac{d}{3a_3^+}-\eps;
\qquad
u_{-2}(\eps)=u_0+\frac{d}{3a_3^+}-\eps;
}
and on the lower boundary:
\eq{211002}{
u_{-1}(-\eps)=u_0-\frac{d}{3a_3^+}+\eps;
\qquad
u_{-2}(-\eps)=u_0+\frac{d}{3a_3^+}+\eps;
}

The isocline $D_+=0$ is a hyperbola:
\eq{290303}{
2x_2D_+(x)=3a_3^+(x_1-x_2-u_0)^2-\frac{d^2+36(a_3^+x_2)^2}{3a_3^+}\,.
}
Accordingly, the roots of the function $D_+$ on the upper boundary are given by
\eq{211003}{
u_{+1}(\eps)=u_0-\frac{\sqrt{d^2+(6a_3^+\eps)^2}}{3a_3^+}+\eps\,;
\qquad
u_{+2}(\eps)=u_0+\frac{\sqrt{d^2+(6a_3^+\eps)^2}}{3a_3^+}+\eps\,;
}
and on the lower boundary:
\eq{211004}{
u_{+1}(-\eps)=u_0-\frac{\sqrt{d^2+(6a_3^+\eps)^2}}{3a_3^+}-\eps\,;
\qquad
u_{+2}(-\eps)=u_0+\frac{\sqrt{d^2+(6a_3^+\eps)^2}}{3a_3^+}-\eps\,.
}

The obvious inequality
$$
d+6a_3^+\eps\ge\sqrt{d^2+(6a_3^+\eps)^2}
$$
implies the following relations:
\eq{211005}{
\begin{aligned}
u_{-1}(\eps)\le u_{+1}(\eps)\qquad&\text{and}\qquad u_{-2}(\eps)\le u_{+2}(\eps);
\\
u_{+1}(-\eps)\le u_{-1}(-\eps)\qquad&\text{and}\qquad u_{+2}(-\eps)\le u_{-2}(-\eps).
\end{aligned}
}

Since the leading coefficient of the polynomials $D_+(u,\eps)$ and $D_-(u,\eps)$ is positive, the inequalities $D_+(u,\eps)>0$ and $D_-(u,\eps)>0$ hold outside the roots of the polynomials, i.\,e., for
\eq{211006}{
u<u_{-1}(\eps)\qquad\text{and}\qquad u>u_{+2}(\eps),
}
while the inequalities $D_+(u,-\eps)>0$ and $D_-(u,-\eps)>0$ hold on the interval between the roots, since the leading coefficient is negative:
\eq{211007}{
u_{-1}(-\eps)<u<u_{+2}(-\eps).
}
This means that we can construct domains with simple foliations
$$
\Om{R}(-\infty,u_{-1}(\eps)),\quad
\Om{L}(u_{-1}(-\eps),u_{+2}(-\eps)),
\quad\text{and}\quad\Om{R}(u_{+2}(\eps),\infty).
$$
To connect the first two domains, we place between them a triangle of bilinearity
$$
\Om{T}(u_{01}),
$$
where the Bellman function is determined by the value on the upper boundary at the point $u_{01}=u_{-1}(\eps)+\eps=u_{-1}(-\eps)-\eps$ and by the quadratic polynomial on the lower boundary according to the formula
\eq{231001}{
\Bell(x)=(x_2+\eps)\frac{f_+(u_{01})-f_-(u_{01})}{2\eps}+f_-
(x_1)+a_2^-(\eps^2-x_2^2).
}
Between the other two domains we insert a fissure
$$
\Om{SW}(v_{-2},u_{+2}),
$$
for which we need to narrow the domain $\Om{L}(u_{-1}(-\eps),u_{+2}(-\eps))$ to $\Om{L}(u_{-1}(-\eps),v_{-2})$. This can be done since we always have $v_{-2}<u_{02}-\eps<u_{+2}(-\eps)$. The resulting foliation~\eqref{161001} is shown in Figure~\ref{241002}.

It should be noted that all these arguments are valid for small $\eps$, more precisely for $\eps\le\eps_1$, where $\eps_1$ is the root of the equation $u_{-1}(-\eps)=v_{-2}(\eps)$, i.\,e., $\eps_1$ is the value of the parameter $\eps$ at which the domain $\Om{L}(u_{-1}(-\eps),v_{-2})$ shrinks to a single segment.

It is not hard to understand what happens for larger $\eps$. The horizontal herringbone will no longer grow from a point on the lower boundary. To find the point from which it can grow, one must consider the integral curve $\ell_\eps$ of our field emerging from the point $(u_{+2},\eps)$, shift it to the right by $\eps$ (obtaining a possible spine of a left herringbone), and find the point $C$ at which this curve intersects the right side of the triangle $\Om{T}(u_{01})$. Such a point $C$ always exists for $\eps>\eps_1$, and its coordinate $C_2=C_2(\eps)$, remaining negative, increases monotonically from the value $C_2(\eps_1)=-\eps_1$.

These assertions are easier to verify if, instead of shifting the integral curve $\ell_\eps$ to the right, we shift the triangle $\Om{T}(u_{01})$ to the left. The advantage of considering this shift is that the line containing the right side of the shifted triangle no longer depends on $\eps$ — it is the line passing through the point $(u_{01},0)$. This is not merely a line independent of $\eps$; it is the line on which $D_-=0$, i.\,e., it bounds the domain where $D_->0$. If we want to make a spine of a herringbone from the integral curve $\ell_\eps$, we cannot continue it beyond this boundary. Note also that at this point the slope of the curve $\ell_\eps$ is equal to one.

Since in the lower half of the strip the curves $\ell_\eps$ increase monotonically with $\eps$, the vertical coordinate of the intersection point of the curve $\ell_\eps$ with the right side of the shifted triangle also increases. Moreover, it becomes obvious that this intersection point remains in the lower half of the strip for all $\eps$. Indeed, the part of the shifted triangle lying in the upper half of the strip is located to the left of the point $(u_{01},0)$, while the upper half of the curve $\ell_\eps$ lies to the right of the point $(u_{02},0)$, where the curve $\ell_\eps$ intersects the midline of the strip. Therefore, there can be no intersection points in the upper half of the strip. To obtain the desired point $C$, we need to shift the resulting intersection point to the right by $\eps$, i.\,e., the coordinates of $C$ satisfy $C_1+C_2=u_{01}+\eps$.

Thus, the last lower rib of the herringbone (the extremal segment going from the point $C$ to the lower boundary) continues the spine of the herringbone in a $C^1$-smooth manner, since we know that at the lower point (the point $C$) the spine has slope $1$. It cuts off from the triangle $\Om{T}(u_{01})$ a triangle near the right vertex; the remaining domain of bilinearity is a right trapezoid (see Fig.~\ref{241004}).
\begin{figure}[h]
    \centering
    \includegraphics[scale = 0.35]{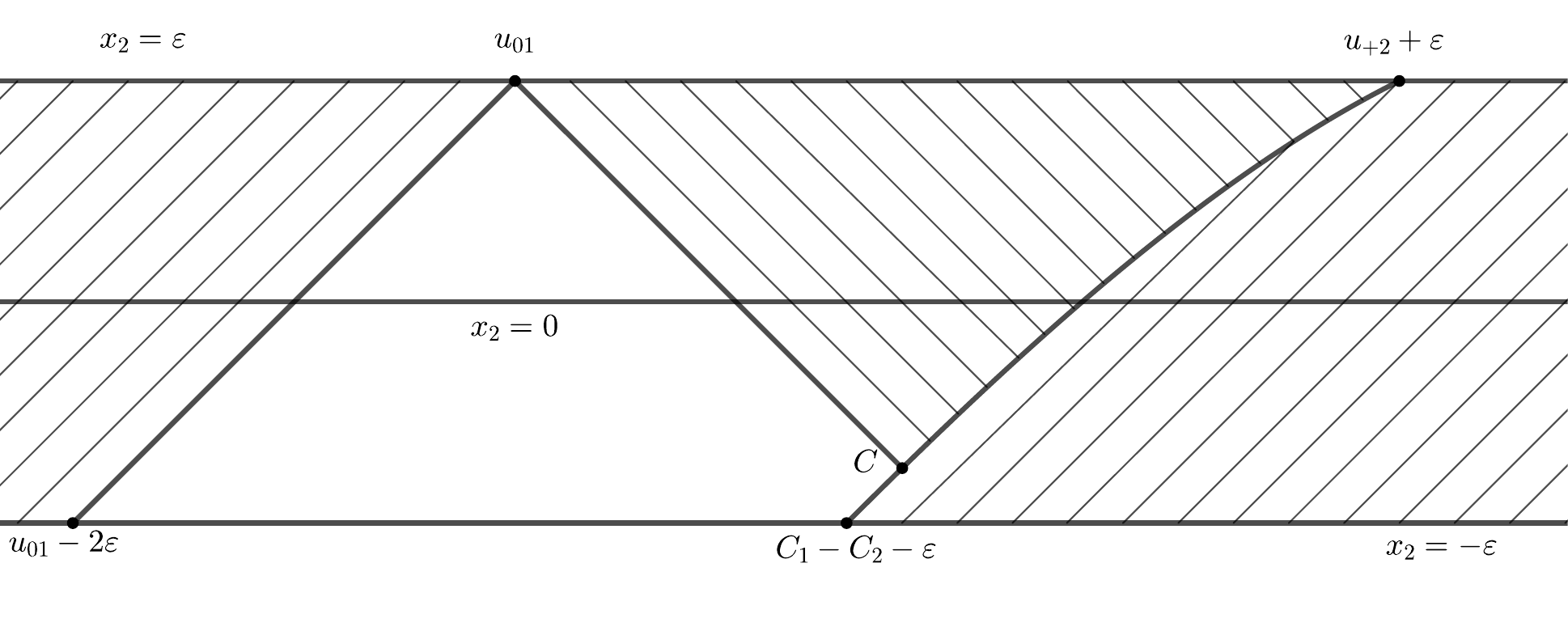}
    \caption{The foliation~\ref{241003}.}
    \label{241004}
\end{figure}
As a result, for $\eps>\eps_1$ we obtain the foliation
\eq{241003}{
\Oe=\Om{R}(-\infty,u_{01}-\eps)\cup\Om{Tz}(u_{01},C)
\cup\Om{HB}\big(C,(u_{+2}\!+\eps,\eps)\big)
\cup\Om{R}(u_{+2},\infty),
}
where
\eq{241005}{
\Om{Tz}(u_{01},C)\df\{x\in\Oe\colon u_{01}-\eps\le x_1-x_2\le C_1-C_2,\;x_1+x_2\le C_1+C_2\},
}
and the function $\Bell$ in this domain is given by formula~\eqref{231001}.
The complete foliation is shown in Figure~\ref{241004}.

Now let us turn to the case where $d=0$ and $a_3^->0$. The local foliation~\eqref{170802a} for small $\eps$ is given by Proposition~\ref{130804}. It is shown in Figure~\ref{170801}.
We shall now show that in fact this foliation extends to the entire strip (i.\,e., $u_1=-\infty$ and $u_2=+\infty$) and remains the same for all $\eps$:
\eq{010302}{
\Oe=\Om{R}(-\infty,v_+^l)\cup\Om{HB}\big((v_+^l\!+\!\eps,\eps),(u_+^r\!\!+\!\eps,\eps)\big)
\cup\Om{R}(u_+^r,+\infty).
}

The fact that the function $2x_2D_+$ is a quadratic polynomial in the first variable guarantees the existence of only two roots $u_+^{l,r}$ of the equation $D_+(u,\eps)=0$, as well as the fact that $D_+(u,\eps)>0$ for $u<u_+^l$ and $u>u_+^r$. The inequality $D_-(u,\eps)>0$ holds throughout the upper half-strip by virtue of the relation
\eq{010303}{
2x_2D_-(x)=3(a_3^+-a_3^-)(x_1+x_2-u_0)^2+12a_3^-x_2^2.
}

Thus, the only thing left to verify is that the integral curve $\ell_\eps$ lies entirely in the domain where $D_+>0$, i.\,e., on the left it reaches the upper boundary at $v_+^l\le u_+^l$. This follows, for example, from the fact that if we move along our integral curve from the point $(u_0,0)$ to the left toward the upper boundary, we cannot cross the isocline $X_0$, which on the interval between $(u_0,0)$ and $(u_+^l,\eps)$ lies below the isocline $X_1$. In our case this can be easily verified by direct computation, since $X_1$ is the line
\eq{020301}{
x_1-u_0=\bigg(1-2\sqrt{\frac{a_3^+}{a_3^+-a_3^-}}\bigg)x_2\,,
}
and $X_0$ is the upper branch of the hyperbola
\eq{020302}{
\Big(x_2+\frac{a_3^++a_3^-}{a_3^+-a_3^-}\Big)^2-\Big(x_1-u_0-\eps\Big)^2
=\frac{4a_3^+a_3^-}{(a_3^+-a_3^-)^2}\,.
}

Finally, the last degenerate case: $d=0$ and $a_3^-=0$. As before, the suitable integral curve $\ell_\eps$ connects the points $(u_+^r,\eps)$ and $(u_0,0)$. Thus, we can construct a herringbone extending from the point $(u_0+\eps,0)$ to the point $(u_+^r+\eps,\eps)$. Incidentally, note that under the conditions $d=0$ and $a_3^-=0$ we have $u_+^r+\eps=3\eps$ and $u_+^l+\eps=-\eps$. Since $2x_2D_-=3a_3^+(x_1+x_2-u_0)^2$, the inequality $D_-\ge0$ holds throughout the upper half-strip, while $D_+(u,\eps)>0$ holds outside the interval between the roots, i.\,e., for $u<-\eps$ and $u>3\eps$. Thus, one can construct two domains with right simple foliation: $\Om{R}(-\infty,u_0-\eps)$ and $\Om{R}(u_0+3\eps,+\infty)$. The herringbone $\Om{HB}\big((u_0+\eps,0),(u_0+4\eps,\eps)\big)$ foliates part of the space between them, and what remains is the trapezoid $\Om{Tz}\big(u_0,(u_0,0)\big)$ (the definition of the domain is given in formula~\eqref{241005}), which is a domain of bilinearity. Altogether, for all $\eps$ we obtain the foliation
\eq{040301}{
\begin{aligned}
\Oe=&\Om{R}(-\infty,u_0-\eps)\cup\Om{Tz}\big(u_0,(u_0,0)\big)\cup
\\
&\cup\Om{HB}\big((u_0+\eps,0),(u_0+4\eps,\eps)\big)
\cup\Om{R}(u_0+3\eps,\infty).
\end{aligned}
}

This foliation is shown in Figure~\ref{130402}.
\begin{figure}[h]
    \centering
    \includegraphics[scale = 0.4]{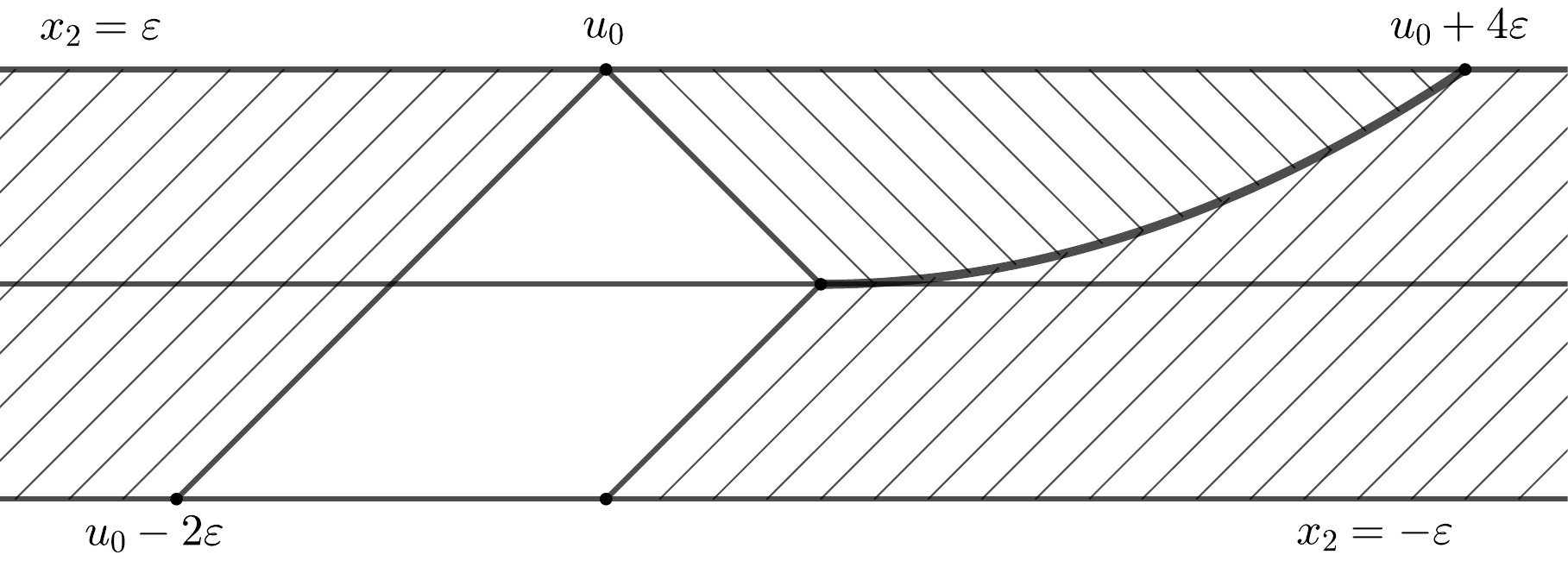}
    \caption{The foliation~\eqref{040301}.}
    \label{130402}
\end{figure}

\section{Survey of foliations\\ 
for polynomials of degree at most three}

In this section we collect together all the results describing the foliations that arise in the cases where the boundary functions have the form
$$
f_\pm(t)=a_3^\pm t^3+a_2^\pm t^2+a_1^\pm t+a_0^\pm.
$$
In order not to increase the number of cases under consideration unnecessarily, we shall not describe those variants that are obtained by symmetry about the vertical and/or horizontal axis. 
If it turns out that $|a_3^+|<|a_3^-|$, we change the sign of $x_2$, i.\,e., interchange $f_+$ and $f_-$. 
If after this we have $a_3^+<0$, we change the sign of $x_1$. Thus we may assume that $a_3^+\ge |a_3^-|$.

\subsection{$a_3^+=a_3^-=0$}\hss

We begin with the case where both polynomials have degree less than three. All variants were described in Section~3 of \cite{First}.

\subsubsection{$a_2^+=a_2^-$}\hss

In this case, for all $\eps$, one of the following three variants occurs:
\begin{enumerate}
    \item [\rm a)] If $a_1^+>a_1^-$, we have a right simple foliation $\Om{R}(-\infty,+\infty)$, shown in Figure~\ref{251001}.
\begin{figure}[h]
    \centering
    \includegraphics[scale = 0.35]{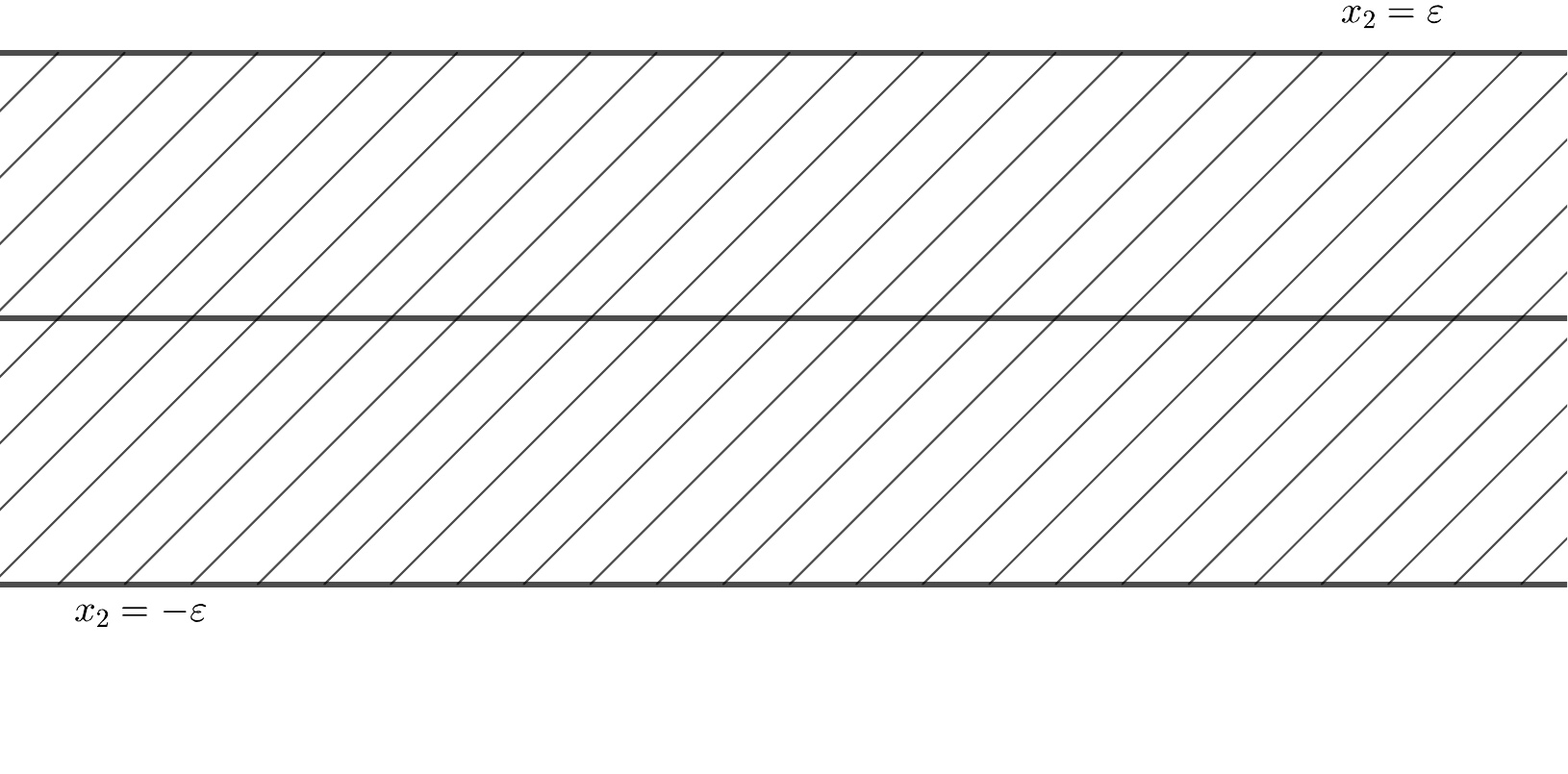}
    \caption{Right simple foliation.}
    \label{251001}
\end{figure}
    \item [\rm b)] If $a_1^+<a_1^-$, we have a left simple foliation $\Om{L}(-\infty,+\infty)$, shown in Figure~\ref{251002}.
\begin{figure}[h]
    \centering
    \includegraphics[scale = 0.36]{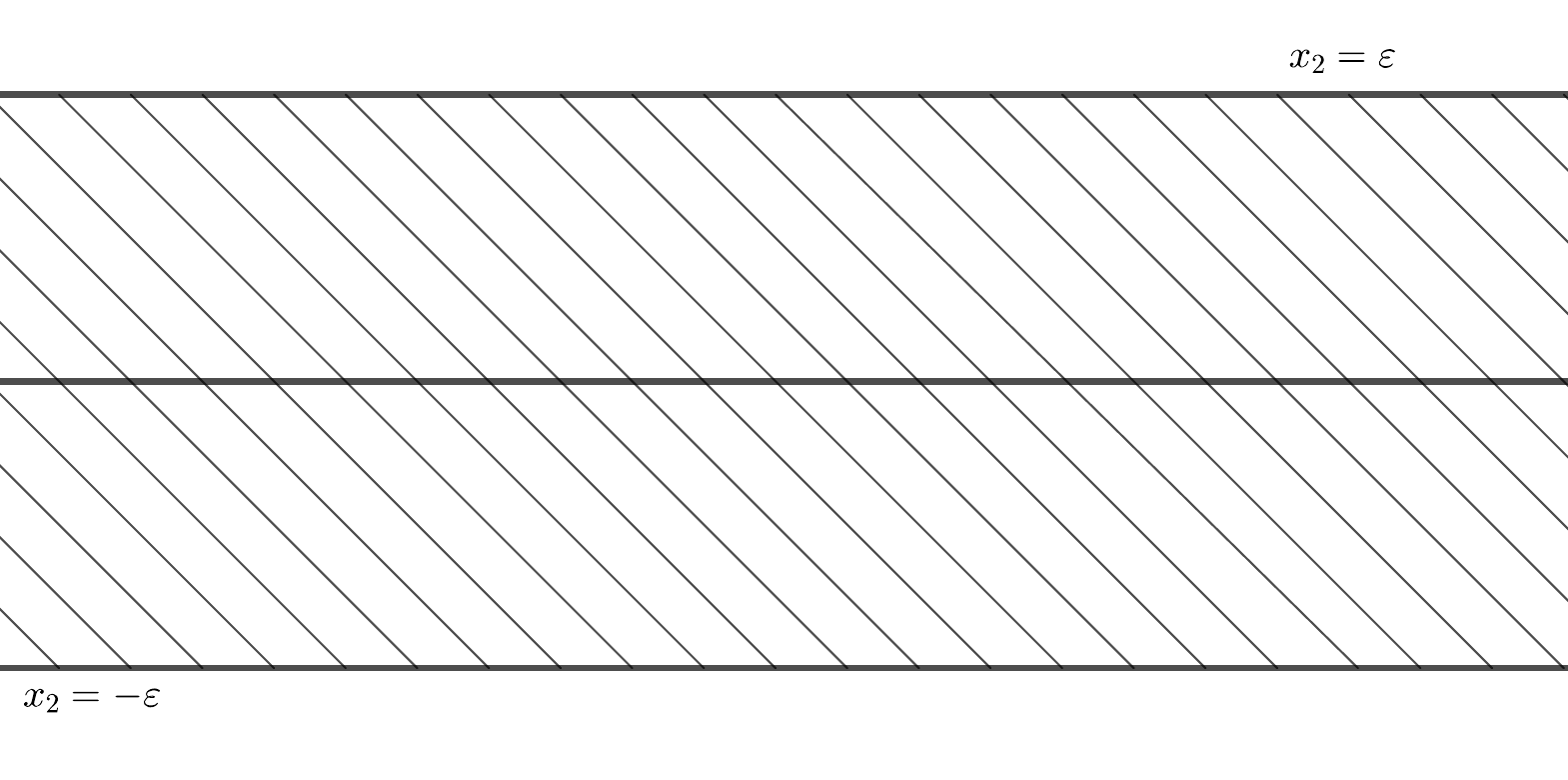}
    \caption{Left simple foliation.}
    \label{251002}
\end{figure}
    \item [\rm c)] If $a_1^+=a_1^-$, the Bellman function is bilinear throughout the strip.
\end{enumerate}

\subsubsection{$a_2^+\ne a_2^-$}\hss

The function $f'_+-f'_-$ is linear with nonzero slope, and it has a root at
\eq{291001}{
u_0=\frac{a_1^--a_1^+}{2(a_2^+-a_2^-)}.
}
\begin{enumerate}
    \item [\rm a)] If $a_2^+<a_2^-$, the foliation has the form
\eq{271001}{
\Om{R}(-\infty,u_0-\eps)\cup\Om{T}(u_0)\cup\Om{L}(u_0+\eps,+\infty).
}
It is shown in Figure~\ref{271002}.

\begin{figure}[h]
    \centering
    \includegraphics[scale = 0.3]{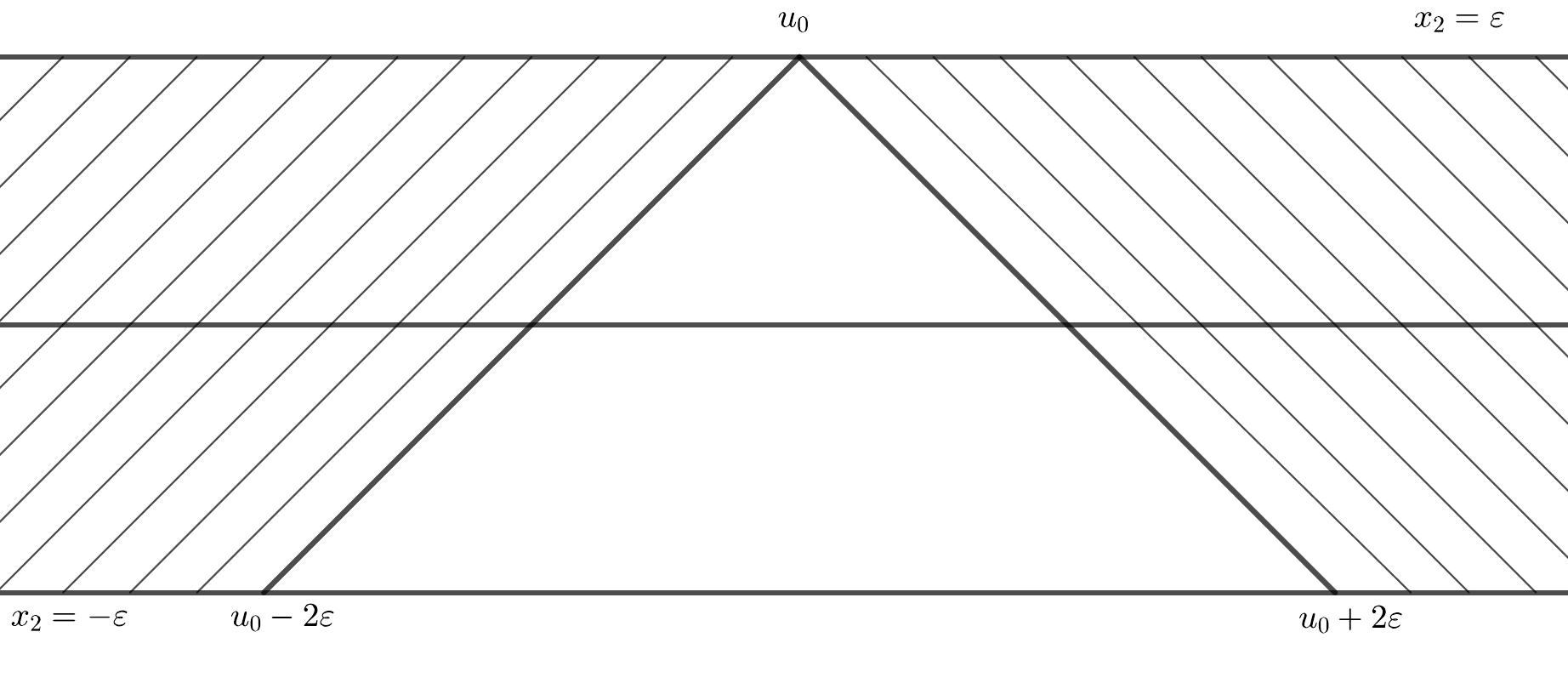}
    \caption{The foliation~\eqref{271001}.}
    \label{271002}
\end{figure}

    \item [\rm b)] If $a_2^+>a_2^-$, the foliation has the form
\eq{271003}{
\Om{L}(-\infty,u_0-\eps)\cup\overline{\Om{T}}(u_0)\cup\Om{R}(u_0+\eps,+\infty).
}
It is shown in Figure~\ref{271004}.

\begin{figure}[h]
    \centering
    \includegraphics[scale = 0.33]{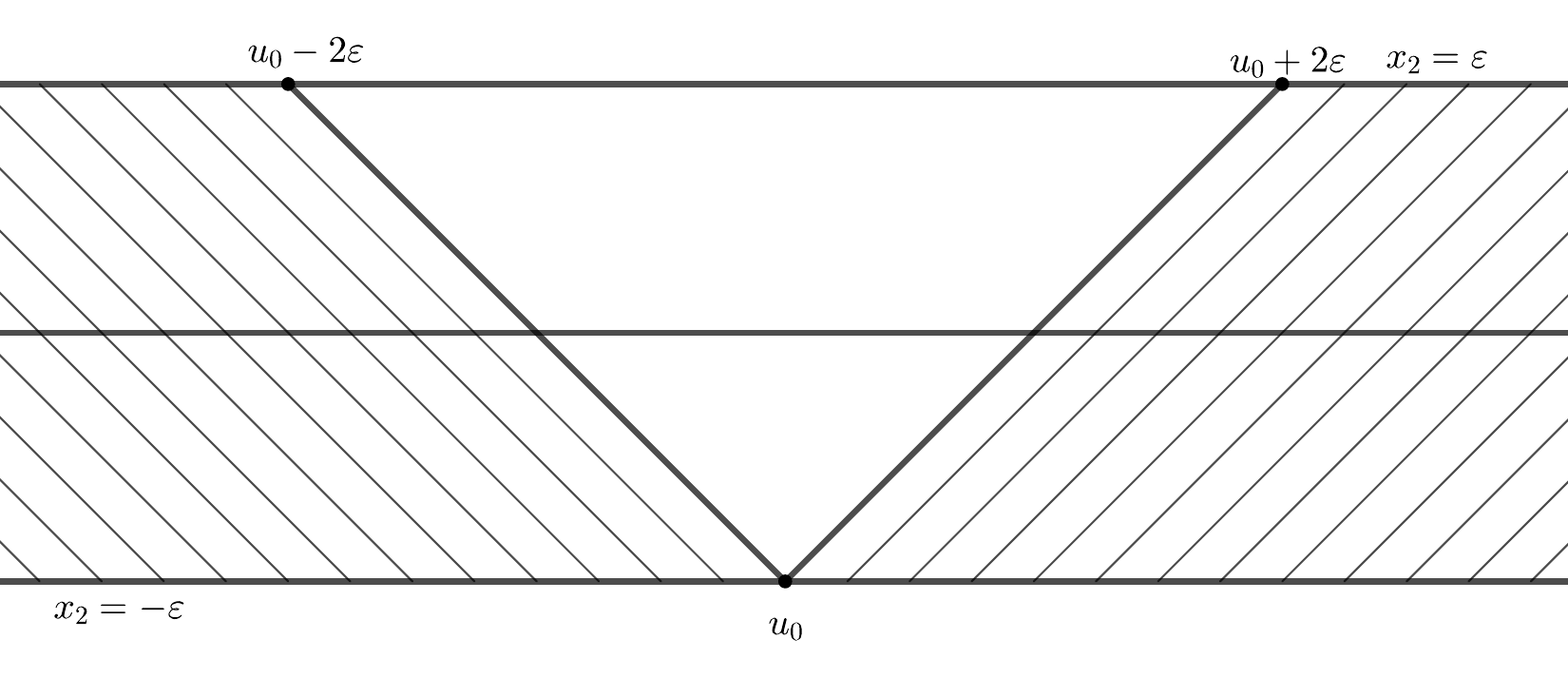}
    \caption{The foliation~\eqref{271003}.}
    \label{271004}
\end{figure}
\end{enumerate}

\subsection{$a_3^+=a_3^-\ne0$}\hss

This case was treated in Section~7 of \cite{First}.

\subsubsection{$a_2^+=a_2^-$}\hss

Set
\eq{281001}{
\eps_0=\sqrt{\frac{|a_1^+-a_1^-|}{12a_3}}.
}
\begin{enumerate}
    \item [\rm a)] If $\eps\le\eps_0$ and $a_1^+>a_1^-$, we have a right simple foliation $\Om{R}(-\infty,+\infty)$, shown in Figure~\ref{251001}.

    \item [\rm b)] If $\eps\le\eps_0$ and $a_1^+<a_1^-$, we have a left simple foliation $\Om{L}(-\infty,+\infty)$, shown in Figure~\ref{251002}.

    \item [\rm c)] If $\eps>\eps_0$, we have an infinite left horizontal herringbone, whose spine is the line
    $$
    x_2=\sign(a_1^+-a_1^-)\eps_0
    $$
    (see Fig.~\ref{281002}).

\begin{figure}[h]
    \centering\hspace{40pt}
    \includegraphics[scale = 0.35]{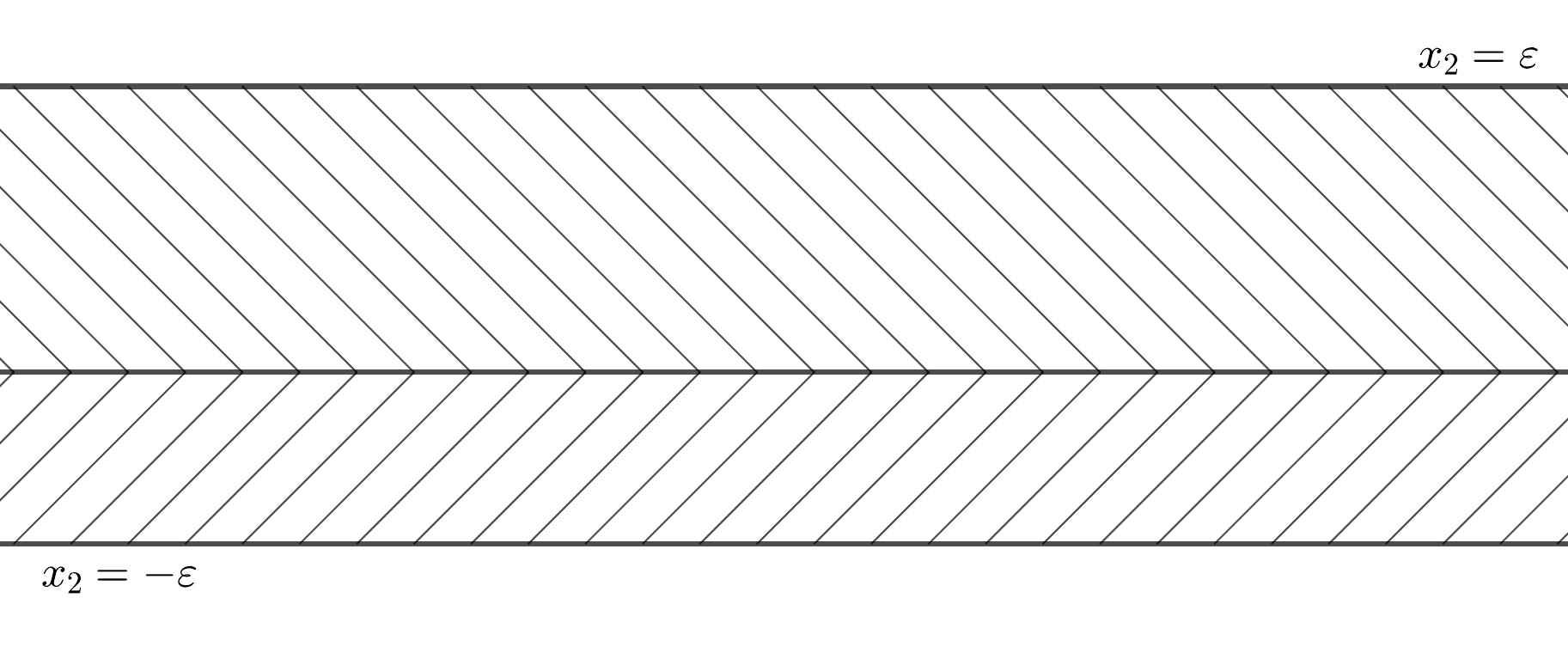}
    \caption{Infinite left horizontal herringbone.}
    \label{281002}
\end{figure}
\end{enumerate}

\subsubsection{$a_2^+\ne a_2^-$}\hss

In this case there is one root $u_0$ of the equation $f'_+=f'_-$ \textup{\bf(}see~\eqref{291001}\textup{\bf)}, and for all $\eps$ we can construct a fissure crossing the midline of the strip at the point $x_1=u_0+\eps$.
\begin{enumerate}
    \item [\rm a)] If $a_2^+>a_2^-$, the foliation has the form
    \eq{301001}{
    \Oe=\Om{L}(-\infty,v_-)\bigcup\Om{SW}(v_-,u_+)\bigcup\Om{R}(u_+,+\infty)\,,
}
where $u_+$ is the root of the equation $D_+(u,\eps)=0$, and $v_-$ is the first coordinate of the point at which the integral curve $\ell_\eps$ of the field from~(5.3), emerging from the point $(u_+,\eps)$, intersects the lower boundary of the strip $\Oe$.
This foliation is shown in Figure~\ref{301002}.
\begin{figure}[h]
    \centering
    \includegraphics[scale = 0.4]{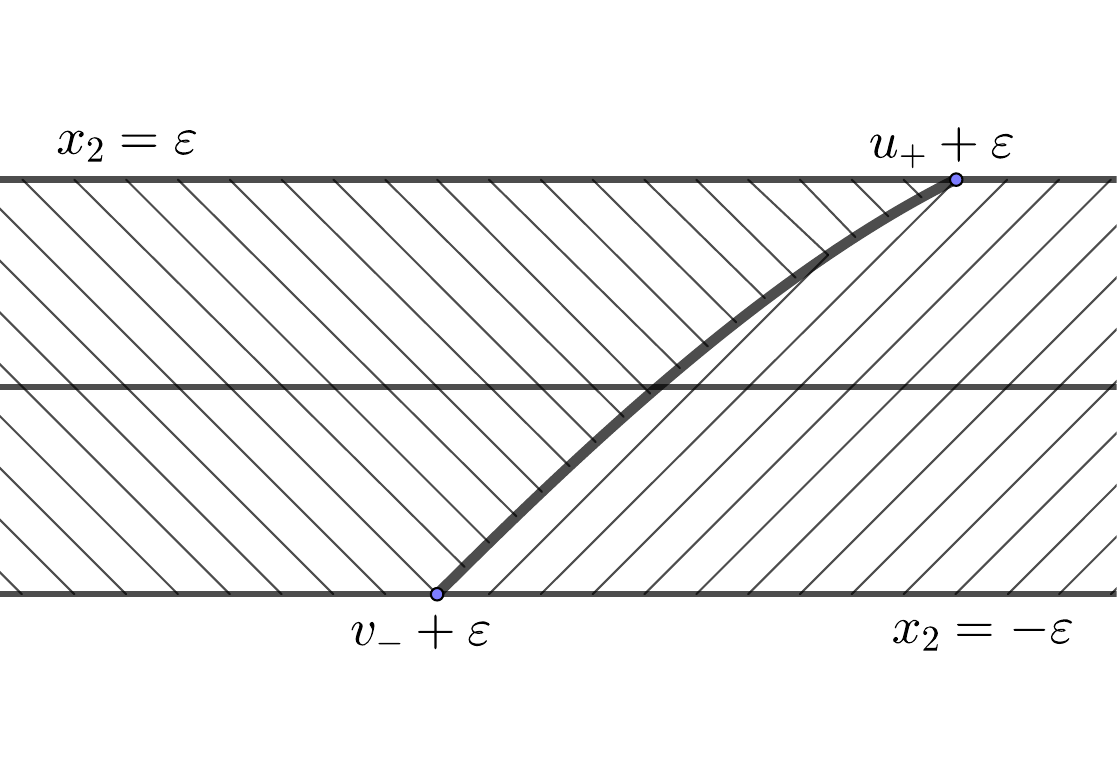}
    \caption{The foliation with an SW-fissure~\eqref{301001}.}
    \label{301002}
\end{figure}

    \item [\rm b)] If $a_2^+<a_2^-$, the foliation has the form
    \eq{301003}{
    \Oe=\Om{R}(-\infty,v_+)\bigcup\Om{NW}(v_+,u_-)\bigcup\Om{L}(u_-,+\infty)\,,
}
where $u_-$ is the root of the equation $D_-(u,-\eps)=0$, and $v_+$ is the first coordinate of the point at which the integral curve $\ell_\eps$ of the field from~(5.3), emerging from the point $(u_-,-\eps)$, intersects the upper boundary of the strip $\Oe$.
This foliation is shown in Figure~\ref{301004}.
\begin{figure}[h]
    \centering
    \includegraphics[scale = 0.35]{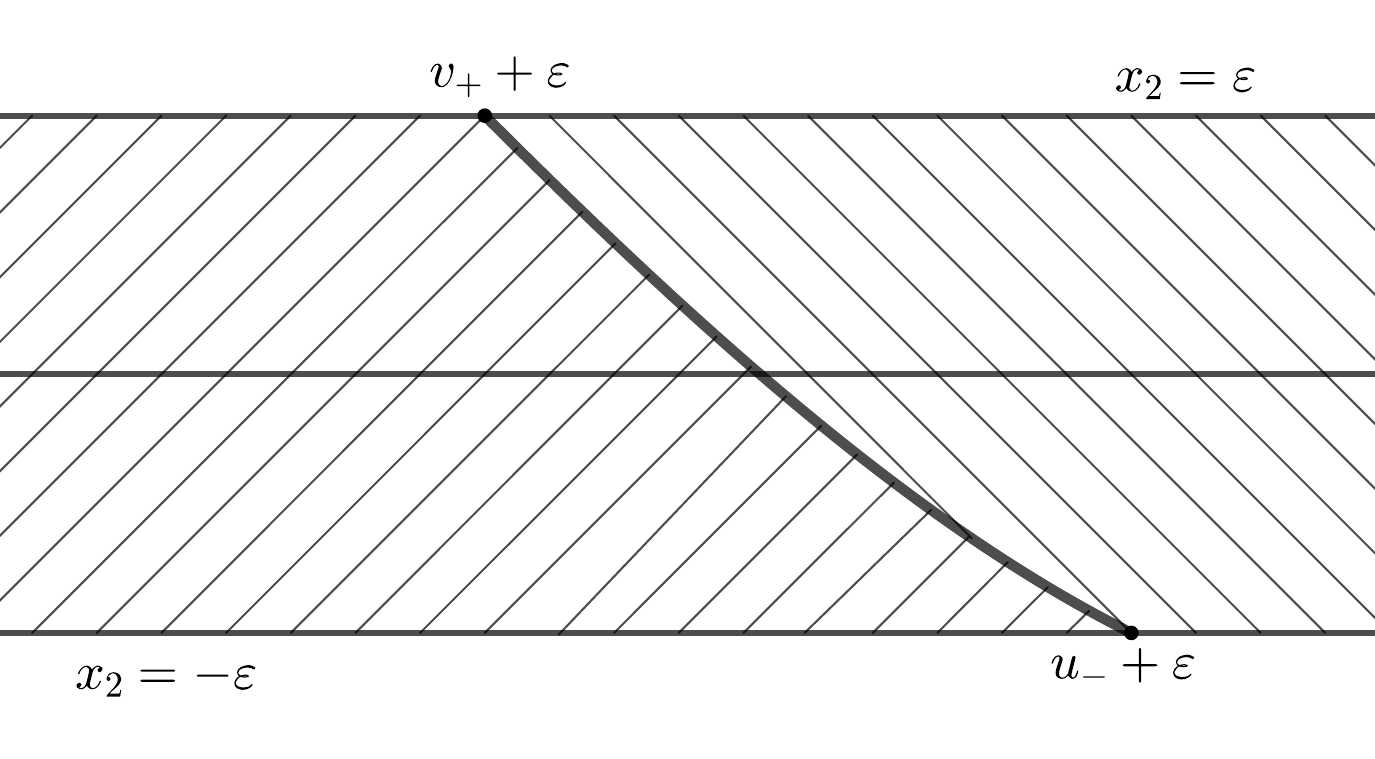}
    \caption{The foliation with an NW-fissure~\eqref{301003}.}
    \label{301004}
\end{figure}
\end{enumerate}

\subsection{$a_3^+\ne a_3^-$, 
$(a_2^+-a_2^-)^2>3(a_3^+-a_3^-)(a_1^+-a_1^-)$}\hss

\subsubsection{$a_3^->0$}\hss

For small $\eps$, the foliation was described in Section~7 (\cite{First}), and for large $\eps$, in Section~\ref{141003}. The critical value $\eps_0$ is the solution of the equation $u_{-1}(\eps)=v_{-2}(\eps)$.
\begin{enumerate}
    \item[\rm a)] If $\eps\le\eps_0$, we have two fissures (see Fig.~\ref{090401})
\eq{020401}{
\begin{aligned}
\Oe&=\Om{R}(-\infty,v_{+1})\cup\Om{NW}(v_{+1},u_{-1})\cup\Om{L}(u_{-1},v_{-2})\cup
\\
&\qquad\cup\Om{SW}(v_{-2},u_{+2})\cup\Om{R}(u_{+2},+\infty)\,,
\end{aligned}
}
\begin{figure}[h]
    \centering
    \includegraphics[scale = 0.45]{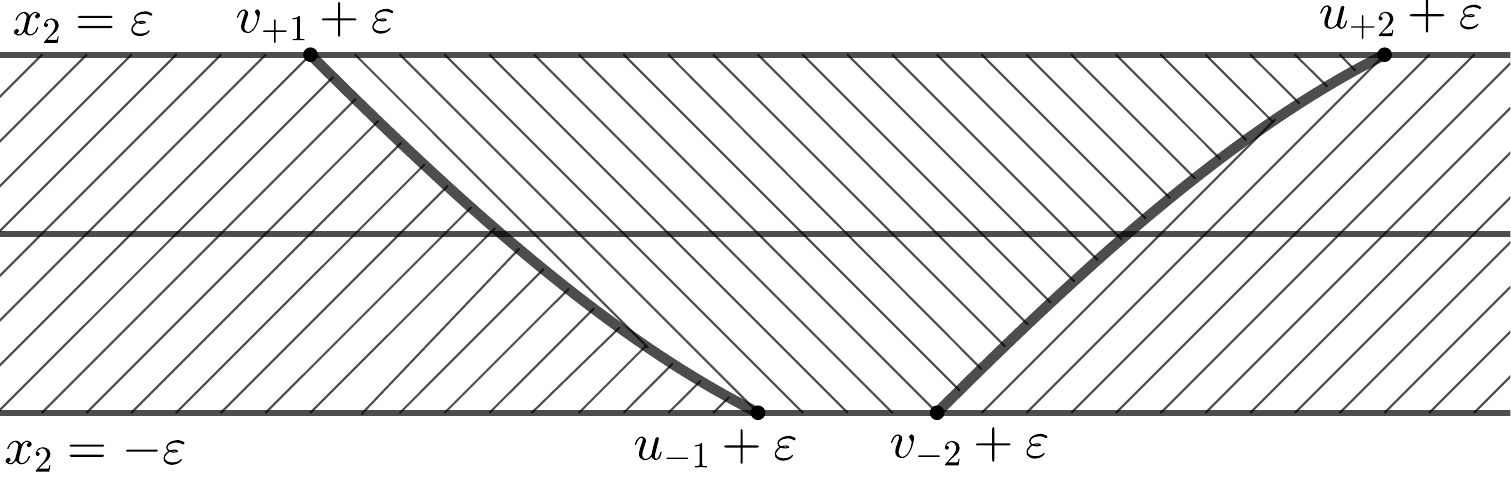}
    \caption{The foliation~\eqref{020401}.}
    \label{090401}
\end{figure}

    \item[\rm b)] If $\eps>\eps_0$, the fissures merge into a major pocket (see Fig.~\ref{090402})
\eq{160202}{
\Oe=\Om{R}(-\infty,v_{+1})\cup\Om{HB}\big((v_{+1}+\eps,\eps),(u_{+2}+\eps,\eps)\big)\cup\Om{R}(u_{2+},+\infty)\,.
}
\begin{figure}[h]
    \centering
    \includegraphics[scale = 0.45]{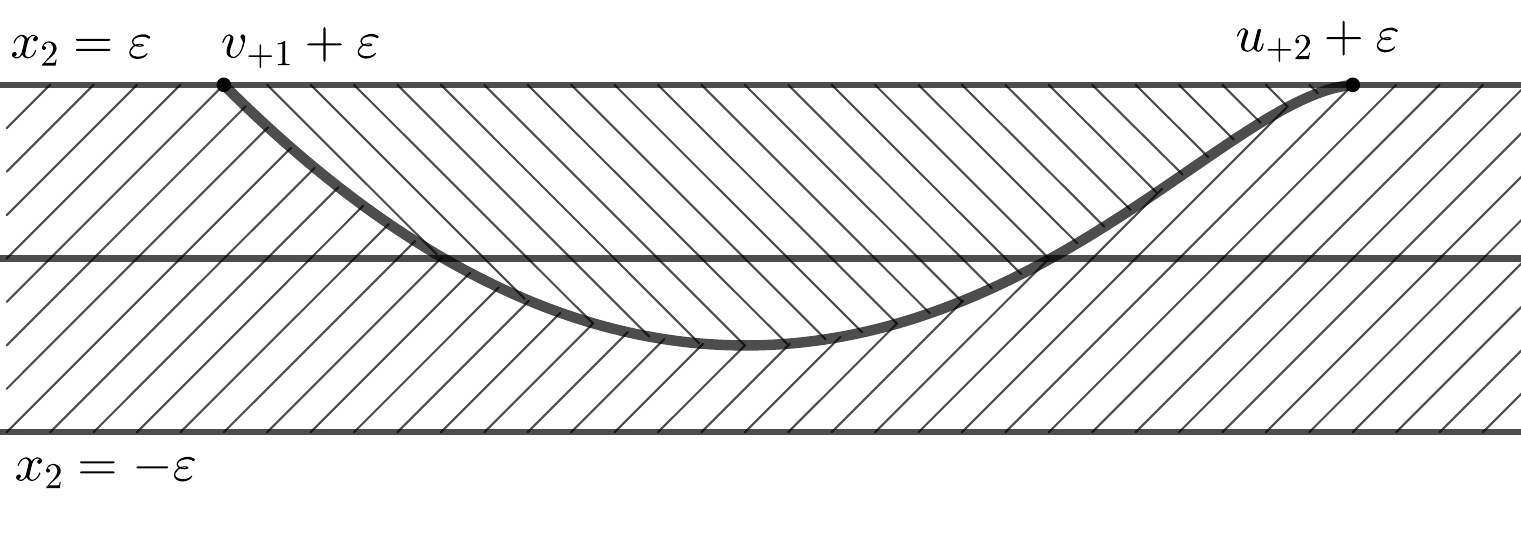}
    \caption{The foliation~\eqref{160202}.}
    \label{090402}
\end{figure}
\end{enumerate}

\subsubsection{$a_3^-=0$}\hss

The foliation was described in Section~\ref{141000}. 
The critical value $\eps_1$ is the root of the equation 
$u_{-1}(-\eps)=v_{-2}(\eps)$.
\begin{enumerate}
    \item[\rm a)] If $\eps\le\eps_1$, we have a triangle of bilinearity and one fissure, described in equality~\eqref{161001}:
$$
\begin{aligned}
\Oe=&\;
\Om{R}(-\infty,u_{01}-\eps)\cup\Om{T}(u_{01})\cup
\Om{L}(u_{01}+\eps,v_{-2})
\\
\quad&\cup
\Om{HB}\big((v_{-2}\!+\eps,-\eps),(u_{+2}\!+\eps,\eps)\big)
\cup\Om{R}(u_{+2},\infty).
\end{aligned}
$$
The foliation is shown in Figure~\ref{241002}.
    \item[\rm b)] If $\eps>\eps_1$, the fissure ``cuts off'' a corner of the triangle, turning it into a trapezoid \textup{\bf(}equality~\eqref{241003}\textup{\bf)}:
$$
\Oe=\Om{R}(-\infty,u_{01}-\eps)\cup\Om{Tz}(u_{01},C)
\cup\Om{HB}\big(C,(u_{+2}\!+\eps,\eps)\big)
\cup\Om{R}(u_{+2},\infty).
$$
\end{enumerate}
The foliation is shown in Figure~\ref{241004}.

\subsubsection{$a_3^-<0$}\hss

For small $\eps$, the foliation was described in Section~7 
(\cite{First}), and for large $\eps$, in Section~12 
(\cite{Second}). The critical value $\eps_2$ is the 
solution of the equation $v_{+1}(\eps)=v_{-2}(\eps)$.

\begin{enumerate}
    \item[\rm a)] If $\eps\le\eps_2$, we have two fissures 
    \textup{\bf(}see formula~(12.2)\footnote{Note that in 
    formula~(12.2) the indices NE and SW of the fissures are interchanged.}\textup{\bf)}:
\eq{040304}{
\begin{aligned}
\Oe=&\Om{R}(-\infty,u_{-1})\cup\Om{NE}(u_{-1},v_{+1})\cup
\Om{L}(v_{+1},v_{-2})\cup
\\
&\cup\Om{SW}(v_{-2},u_{+2})\cup\Om{R}(u_{+2},+\infty)\,.
\end{aligned}
}
The foliation is shown in Figure~\ref{100401}.
\begin{figure}[h]
    \centering
    \includegraphics[scale = 0.35]{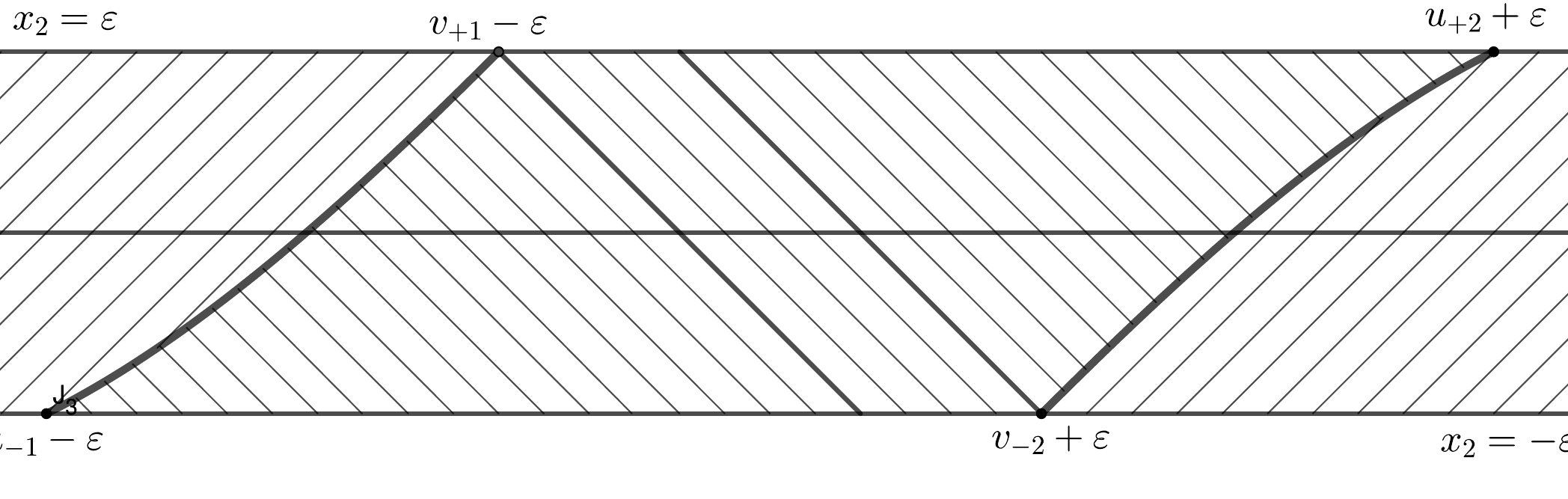}
    \caption{The foliation~\eqref{040304}.}
    \label{100401}
\end{figure}

    \item[\rm b)] If $\eps>\eps_2$, a rectangle of bilinearity appears between the fissures. The foliation is described in equalities~(12.3) or~(12.4):
\eq{040305}{
\begin{aligned}
\Oe=\Om{R}(-\infty,u_{-1})&\cup\Om{HB}\big((C_1-\eps,-C_2),(u_{-1}\!\!-\eps,-\eps)\big)\:\cup
\\
\cup\:\Om{Rect}(C)&\cup\Om{HB}\big((C_1+\eps,C_2),(u_{+2}\!+\eps,\eps)\big)\cup\Om{R}(u_{+2},+\infty)\,,
\end{aligned}
}
\end{enumerate}
The foliation is shown in Figure~\ref{110401}.
\begin{figure}[h]
    \centering
    \includegraphics[scale = 0.3]{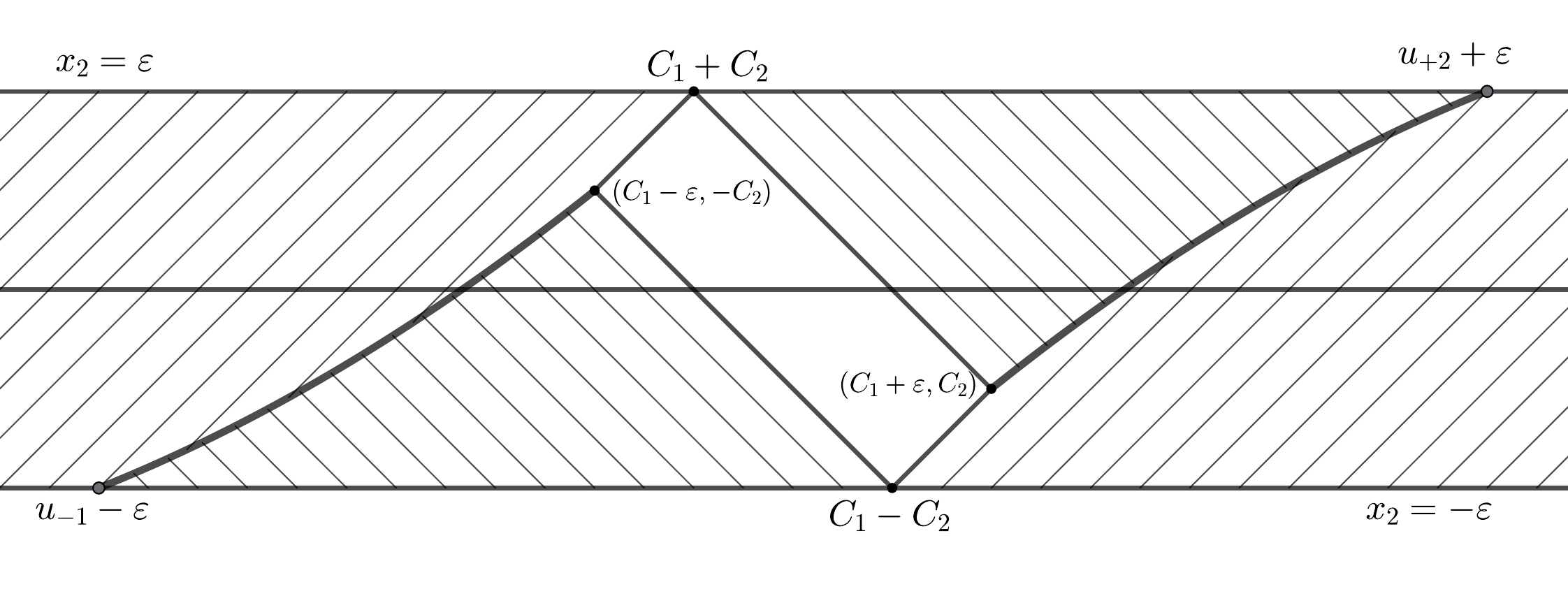}
    \caption{The foliation~\eqref{040305}.}
    \label{110401}
\end{figure}

Here $C=C(\eps)$ is the intersection point of the curves $\ell_\eps^+$ 
and $\bar\ell_\eps^-$, where $\ell_\eps^+$ is the integral 
curve of the field from~(5.3) starting from the point $(u_{+2},\eps)$ with 
positive slope (i.\,e., the curve that generates the herringbone 
on the right), and $\bar\ell_\eps^-$ is the curve symmetric with respect to the $x_1$-axis 
to the integral curve $\ell_\eps^-$ of the field from~(8.22), which starts from the point $(u_{-1},-\eps)$ with positive 
slope (i.\,e., the curve that generates the herringbone on the left).

\subsection{$a_3^+\ne a_3^-$, 
$(a_2^+-a_2^-)^2<3(a_3^+-a_3^-)(a_1^+-a_1^-)$}\hss

\subsubsection{$a_3^-\ge0$}\hss

This case was treated in Section~9 \textup{\bf(}see~\cite{Second}\textup{\bf)}. 
Formula~(9.2) defined the value of the parameter $\eps_0^+$,
at which the foliation changes.

\begin{enumerate}
    \item[\rm a)] If $\eps\le\eps_0^+$, the entire strip is filled with a right
    simple foliation:
\eq{050301}{
\Oe=\Om{R}(-\infty,+\infty)\,.
}
See Fig.~\ref{251001}.
    \item[\rm b)] For $\eps>\eps_0^+$, a minor pocket appears on the
    upper boundary and the foliation takes the form:
\eq{050302}{
\Oe=\Om{R}(-\infty, v_+)\cup\Om{HB}\big((v_+,\eps),(u_+^r,\eps)\big)\cup\Om{R}(u_+^r,+\infty)\,.
}
The foliation is shown in Figure~\ref{110402}.
\begin{figure}[h]
    \centering
    \includegraphics[scale = 0.25]{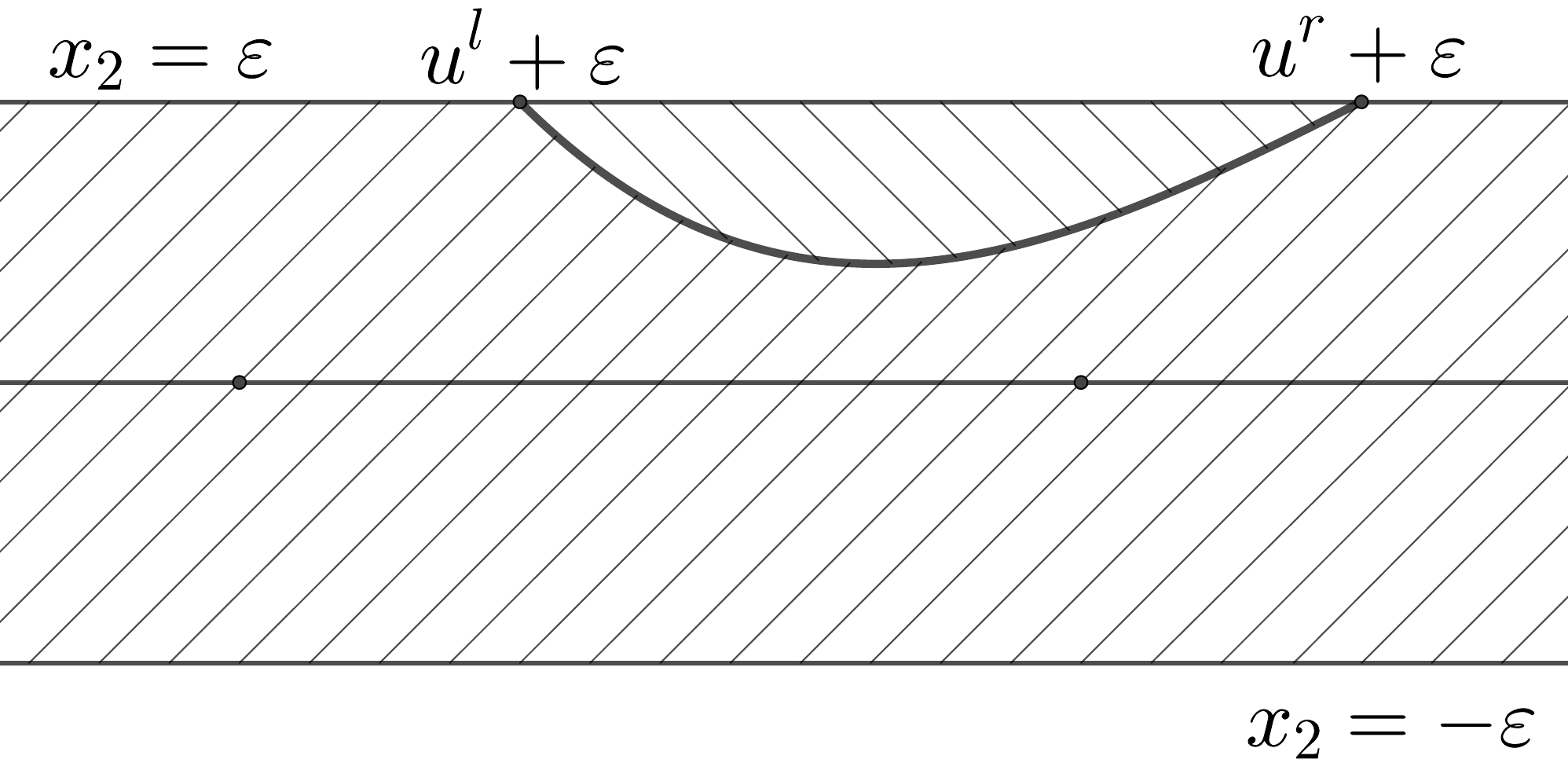}
    \caption{The foliation~\eqref{050302}.}
    \label{110402}
\end{figure}
\end{enumerate}

\subsubsection{$a_3^-<0$}\hss

This case was treated in Section~10 \textup{\bf(}see~\cite{Second}\textup{\bf)}. 
Formula~(10.1) defined the value of the parameter $\eps_0^-$,
when the foliation changes once more. The final change of the foliation occurs at $\eps=\eps_2$.

For small $\eps$, the foliation develops in the same way as in the previous case.
\begin{enumerate}
    \item[\rm a)] If $\eps\le\eps_0^+$, the entire strip is filled with a right
    simple foliation~\eqref{050301}. See Fig.~\ref{251001}.
    \item[\rm b)] If $\eps_0^+<\eps\le\eps_0^-$, a minor pocket appears on the upper boundary and the foliation has the form~\eqref{050302}. See Fig.~\ref{110402}.
\end{enumerate}    
Subsequently, the situation may develop differently depending on the relative position of the second end $v_+(\eps)$ of the extremal $\ell_\eps^+$ and the root $u_0^-$ of the equation $D_-(u,\eps)=0$ appearing on the upper boundary:
\eq{050303}{
u_0^-=-\eps_0^--\frac{a_2^+-a_2^-}{3(a_3^+-a_3^-)}\,.
}
\begin{enumerate}
    \item[\rm 1.] If $v_+(\eps_0^-)\ge u_0^-$, then 
 \begin{enumerate}
    \item[\rm c1)] up to the moment $\eps_2$ defined by the condition $v_+(\eps_2)=u_-^l(\eps_2)$, the foliation~\eqref{050302} remains in force,
    \item[\rm d1)] and for $\eps>\eps_2$, a second (now right) herringbone appears, together with a rectangle of bilinearity separating the two herringbones:
\eq{070301}{
\begin{aligned}
\Oe=\Om{R}(-\infty,u_-^l)&\cup\Om{HB}\big((C_1-\eps,-C_2),(u_-^l-\eps,-\eps)\big)\:\cup
\\
\cup\:\Om{Rect}(C)&\cup\Om{HB}\big((C_1+\eps,C_2),(u_+^r+\eps,\eps)\big)\cup\Om{R}(u_+^r,+\infty)\,.
\end{aligned}
}
\end{enumerate}
The foliation is shown in Figure~\ref{120401}.
\begin{figure}[h]
    \centering
    \includegraphics[scale = 0.37]{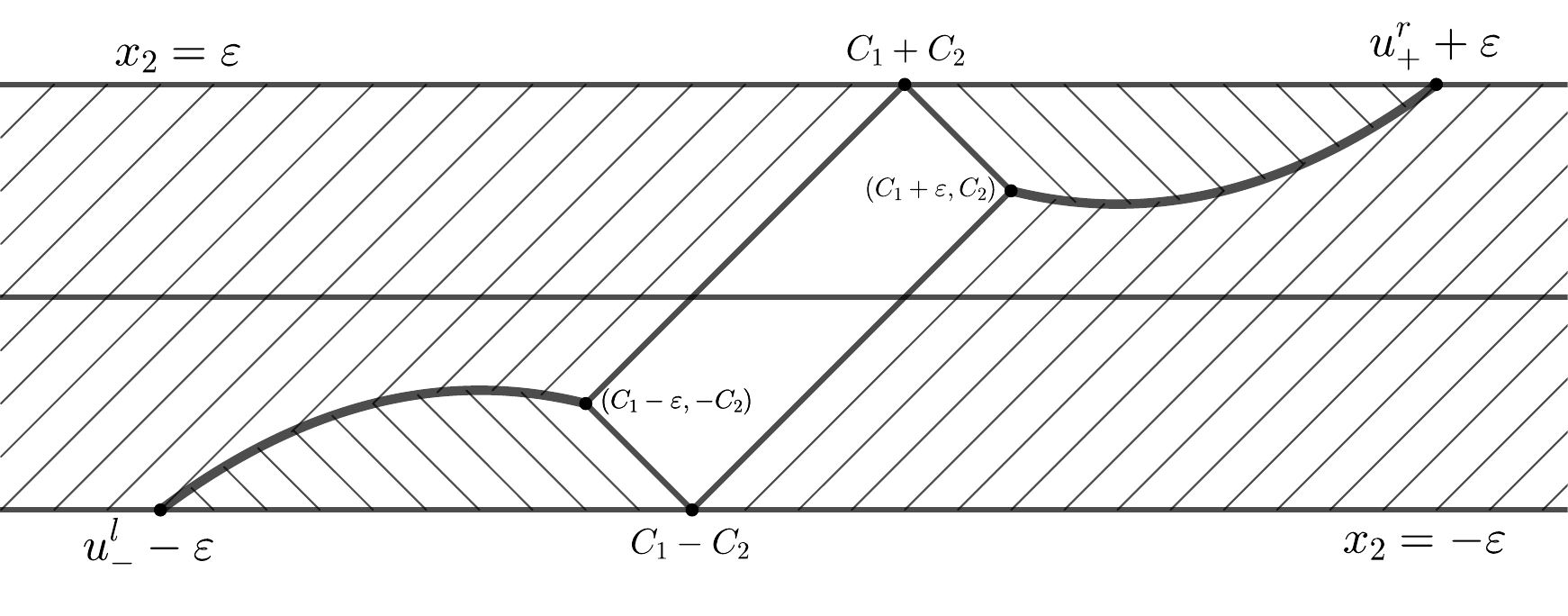}
    \caption{The foliation~\eqref{070301}.}
    \label{120401}
\end{figure}

Here $C=C(\eps)$ is the intersection point of the curves $\ell_\eps^+$ 
and $\bar\ell_\eps^-$, where $\ell_\eps^+$ is the integral 
curve of the field from~(5.3) starting from the point $(u_+^r,\eps)$ with 
positive slope (i.\,e., the curve that generates the herringbone 
on the right), and $\bar\ell_\eps^-$ is the curve symmetric with respect to the $x_1$-axis to the 
integral curve $\ell_\eps^-$ of the field from~(8.22), which starts from the point $(u_-^l,-\eps)$ with positive 
slope (i.\,e., the curve that generates the herringbone on the left).

   \item[\rm 2.] If $v_+(\eps_0^-)< u_0^-$, then 
 \begin{enumerate}
    \item[\rm c2)] at $\eps=\eps_0^-$, a right herringbone appears on the lower boundary, which exists simultaneously with the left herringbone on the upper boundary up to the moment $\eps_2$ defined by the condition $v_+(\eps_2)=v_-(\eps_2)$,
\eq{070302}{
\begin{aligned}
\Oe=\Om{R}(-\infty,u_-^l)&\cup\Om{HB}\big((v_-\!\!-\eps,-\eps),(u_-^l\!\!-\eps,-\eps)\big)\:\cup
\\
\cup\:\Om{R}(v_-,v_+)&\cup\Om{HB}\big((v_+\!+\eps,\eps),(u_+^r\!+\eps,\eps)\big)\cup\Om{R}(u_+^r,+\infty)\,.
\end{aligned}
}
The foliation is shown in Figure~\ref{120402}.
\begin{figure}[h]
    \centering
    \includegraphics[scale = 0.3]{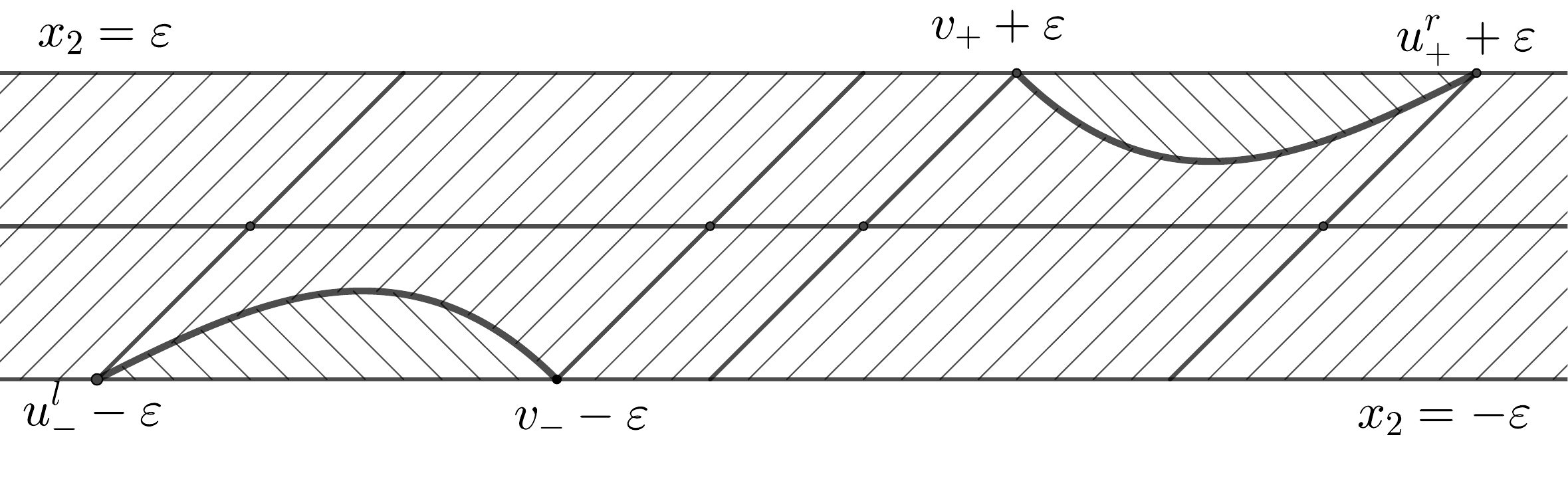}
    \caption{The foliation~\eqref{070302}.}
    \label{120402}
\end{figure}
    \item[\rm d2)] at $\eps=\eps_2$ the domain $\Om{R}(v_-,v_+)$ collapses, and for $\eps>\eps_2$ a rectangle of bilinearity separating the two herringbones appears, with the foliation taking the form~\eqref{070301} (Fig.~\ref{120401}).
 \end{enumerate}
\end{enumerate}

\subsection{$a_3^+\ne a_3^-$, 
$(a_2^+-a_2^-)^2=3(a_3^+-a_3^-)(a_1^+-a_1^-)$}\hss

\subsubsection{$a_3^->0$}\hss 

For all $\eps$, the foliation~\eqref{010302} holds, which is called a medium pocket. It is shown in Figure~\ref{130401}.
\begin{figure}[h]
    \centering
    \includegraphics[scale = 0.3]{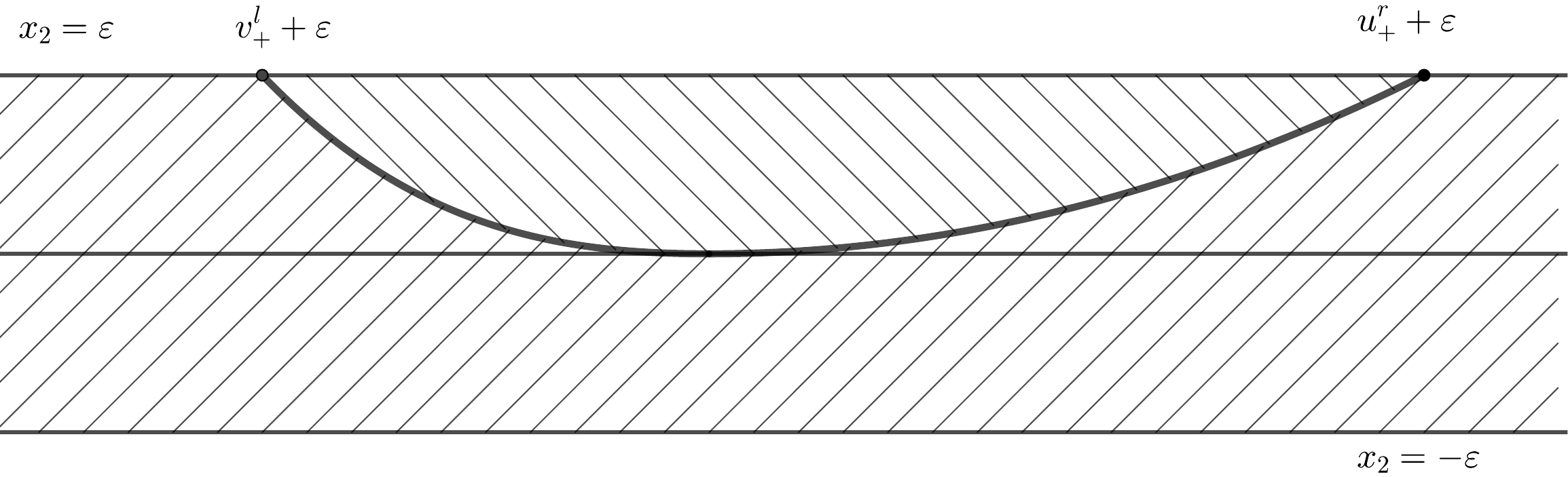}
    \caption{The foliation~\eqref{010302}.}
    \label{130401}
\end{figure}

\subsubsection{$a_3^-=0$}\hss 

For all $\eps$, the foliation~\eqref{040301} holds, which is shown in Figure~\ref{130402}.

\subsubsection{$a_3^-<0$}\hss 

This case was already treated in Section~12 (see~\cite{Second}), where it was shown that for all $\eps$ the foliation
\eq{070303}{
\begin{aligned}
\Oe=\Om{R}(-\infty,u_{-1})&\cup\Om{HB}\big((u_0-\eps,0),(u_{-1}\!\!-\eps,-\eps)\big)\cup\Om{Rect}\big((u_0,0)\big)\cup
\\
&\cup\Om{HB}\big((u_0+\eps,0),(u_{+2}\!+\eps,\,\eps)\big)\cup\Om{R}(u_{+2},+\infty)\,.
\end{aligned}
}
holds. This foliation is shown in Figure~\ref{140802}.

\section*{Acknowledgments}
The author is grateful to P.~B.~Zatitskii, whose critical remarks helped to improve the present text.

\vskip40pt

\section*{Formulas and statements from \cite{First} and \cite{Second}\\
used in the present paper}
$$
T'=\frac{(\eps-T)D\ut{L}_--(\eps+T)D\ut{L}_+}{(\eps-T)D\ut{L}_-+(\eps+T)D\ut{L}_+};
\leqno(4.32)
$$

$$
\begin{cases}
\dot x_1=\eps(f_+'-f_-')-x_2(\eps-x_2)f_-''-x_2(\eps+x_2)f_+'',
\\
\dot x_2=x_2(f_-'-f_+')-x_2(\eps-x_2)f_-''+x_2(\eps+x_2)f_+'';
\end{cases}
\leqno(5.3)
$$

$$
J(x)\!=\!
\begin{pmatrix}
\scriptstyle\eps(f_+''-f_-'')-x_2(\eps-x_2)f_-'''-x_2(\eps+x_2)f_+'''&
\scriptstyle 2x_2(f_-''-f_+'')+x_2(\eps-x_2)f_-'''-x_2(\eps+x_2)f_+'''
\\
\scriptstyle x_2(f_-''-f_+'')-x_2(\eps-x_2)f_-'''+x_2(\eps+x_2)f_+'''&
\scriptstyle(f_-'-f_+')-(\eps-x_2)(f_-''-x_2f_-''')+(\eps+x_2)(f_+''+x_2f_+''')
\end{pmatrix};
\leqno(5.5)
$$


$$
\begin{cases}
\dot x_1=(\eps+x_2)\big[\half(f_-'-f_+')-x_2f_+''\big]
+(\eps-x_2)\big[\half(f_-'-f_+')-x_2f_-''\big],
\\
\dot x_2=(\eps+x_2)\big[\half(f_-'-f_+')-x_2f_+''\big]
-(\eps-x_2)\big[\half(f_-'-f_+')-x_2f_-''\big],
\end{cases}
\leqno(8.22)
$$
\bigskip

{\bf Proposition 8.8.}

Let the function $D_+$ vanish at the point $(u,\eps)$. Assume that
$$
\kappa_+^{\LL}\df\frac{2f'''_+(u+\eps)}{f''_+(u+\eps)-f''_-(u-\eps)-2\eps f'''_+(u+\eps)}>0\,. 
\leqno(8.20)
$$
Then, while the curve $\ell_\eps$ goes inside the domain where $x_2D_+D_->0$,
its slope is strictly between $-1$ and $1$ (i.\,e., it can be considered as a graph of a function of the argument $x_1$). Moreover, for $\eps'<\eps''$ the curve $\ell_{\eps'}$ goes strictly above the curve $\ell_{\eps''}$ while they are inside this domain.
\bigskip

{\bf Remark 8.12.}

In the proposition just proved, it was assumed that we consider the curve $\ell_\eps$ only in the upper or lower half of the strip. However, it is not difficult to extend this assertion to the other half if the integral curve $\ell_\eps$ crosses the midline of the strip at the point $(u_0,0)$, where $u_0$ is a simple root of the equation $f'_+(u)=f'_-(u)$. In this case, upon passing to the other half of the strip, the relations over or under are reversed. The general rule can be formulated as follows: the larger $\eps$, the closer the curve $\ell_\eps$ is to the midline of the strip, $x_2=0$.

$$
\eps_0=\eps_0^+\df\sqrt{\frac1{12a_3^+}
\Big[a_1^+-a_1^--\frac{(a_2^+-a_2^-)^2}{3(a_3^+-a_3^-)}\Big]}\,.
\leqno(9.2)
$$

$$
2x_2D_+(x)=3(a_3^+-a_3^-)(x_1-x_2)^2+2(a_2^+-a_2^-)(x_1-x_2)+(a_1^+-a_1^-)-12a_3^+x_2^2;
\leqno(9.3)
$$

$$
2x_2D_-(x)=3(a_3^+-a_3^-)(x_1+x_2)^2+2(a_2^+-a_2^-)(x_1+x_2)+(a_1^+-a_1^-)+12a_3^-x_2^2;
\leqno(9.4)
$$

$$
\eps_0^-=\sqrt{\frac1{12|a_3^-|}\Big[a_1^+-a_1^--\frac{(a_2^+-a_2^-)^2}{3(a_3^+-a_3^-)}\Big]}=
\sqrt{\frac{a_3^+}{|a_3^-|}}\eps_0^+\,.
\leqno(10.1)
$$

\vskip 50pt

St.-Petersburg Department of Steklov Mathematical Institute

St.-Petersburg State University

\bigskip

vasyunin@pdmi.ras.ru
\end{document}